\documentclass[11pt,a4paper]{article}
\usepackage{amsmath,amssymb,amsthm,amscd,mathrsfs}
\usepackage{indentfirst}
\usepackage[top=25mm, bottom=30mm, left=30mm, right=30mm]{geometry}
\usepackage{color}
\usepackage{comment}
\usepackage{bm}
\usepackage[all]{xy}

\newtheorem{df}{\bf Definition}[section]
\newtheorem{thm}[df]{\bf Theorem}
\newtheorem{lem}[df]{\bf Lemma}

\newtheorem{prop}[df]{\bf Proposition}
\newtheorem{rem}[df]{\bf Remark}
\newtheorem{exa}[df]{\bf Example}
\newtheorem{assum}[df]{\bf Assumption}

\newtheorem*{claim}{\bf Claim}

\newcommand{\R}{\mathbb{R}}
\newcommand{\C}{\mathbb{C}}
\newcommand{\Z}{\mathbb{Z}}

\newcommand{\N}{\mathbb{N}}
\newcommand{\B}{\mathbb{B}}
\newcommand{\K}{\mathbb{K}}

\newcommand{\T}{\mathbb{T}}

\newcommand{\rM}{\mathrm{M}}

\newcommand{\frakS}{\mathfrak{S}}

\newcommand{\cK}{\mathcal{K}}
\newcommand{\cA}{\mathcal{A}}
\newcommand{\cF}{\mathcal{F}}
\newcommand{\cU}{\mathcal{U}}
\newcommand{\cC}{\mathcal{C}}
\newcommand{\cG}{\mathcal{G}}
\newcommand{\cI}{\mathcal{I}}
\newcommand{\cB}{\mathcal{B}}

\newcommand{\act}{\curvearrowright}

\newcommand{\ri}{\mathrm{i}}
\newcommand{\actson}{\curvearrowright}

\newcommand{\sym}{\mathrm{sym}}
\newcommand{\full}{\mathrm{full}}
\newcommand{\anti}{\mathrm{anti}}
\newcommand{\free}{\mathrm{free}}

\newcommand{\sfX}{\mathsf{X}_\full}
\newcommand{\sfY}{\mathsf{X}_\sym}
\newcommand{\sfZ}{\mathsf{X}_\anti}

\newcommand{\Ad}{\operatorname{Ad}}
\newcommand{\Prob}{\operatorname{Prob}}
\newcommand{\id}{\text{\rm id}}

\newcommand{\ovt}{\mathbin{\overline{\otimes}}}

\newcommand{\otm}{\otimes_{\rm min}}
\newcommand{\ota}{\otimes_{\rm alg}}

\newcommand{\supp}{\mathrm{supp}}

\title{\bf Note on bi-exactness for creation operators on Fock spaces}
\author{Kei Hasegawa\and 
Yusuke Isono\thanks{Research Institute for Mathematical Sciences, Kyoto University, 606-8502, Kyoto, Japan \protect \\  E-mail: \texttt{isono@kurims.kyoto-u.ac.jp} \protect \\  
YI is supported by JSPS KAKENHI Grant Number 20K14324.}\and
Tomohiro Kanda}
\date{}
\begin{document}
\maketitle

\begin{abstract}
	In this note, we introduce and study a notion of bi-exactness for creation operators acting on full, symmetric and anti-symmetric Fock spaces. This is a generalization of our previous work, in which we studied the case of anti-symmetric Fock spaces. As a result, we obtain new examples of solid actions as well as new proofs for some known solid actions. We also study free wreath product groups in the same context.
\end{abstract}

\section{Introduction}

This article is a complementary note of our previous work \cite{HIK20}, in which we studied some boundary amenability phenomenon for creation operators on anti-symmetric Fock spaces. 
Our work was inspired by Ozawa's work on Bernoulli actions and wreath product groups \cite{Oz04,BO08}, which we briefly explain here.

Let $\Lambda$ be an amenable group and $\Gamma$ an exact group. Let $\Lambda \wr \Gamma = \left[\bigoplus_\Gamma \Lambda \right]\rtimes \Gamma$ be the wreath product group. Ozawa constructed a C$^*$-subalgebra $C(X) \subset \ell^\infty(\bigoplus_\Gamma \Lambda)$ such that, with the quotient $C(\partial X):=C(X)/c_0(\bigoplus_\Gamma \Lambda)$, (i) the left and right translation of $\bigoplus_\Gamma \Lambda$ on $C(\partial X)$ is trivial, and (ii) the left and right translation of $\Gamma\times\Gamma$ on $C(\partial X)$ is \textit{topologically amenable}. This amenability implies that the following map 
	$$ C(X) \ota C_\lambda^*(\Lambda \wr \Gamma)\ota C_\rho^*(\Lambda \wr \Gamma) \to \B(\ell^2(\Lambda \wr \Gamma))$$
arising from inclusions is a minimal norm bounded $\ast$-homomorphism up to some \textit{relative} compact operators. Note that if there is no such compact operators, this boundedness is equivalent to amenability of $\Lambda \wr \Gamma$. 
Ozawa used the boundedness to deduce some  \textit{rigidity} of $L(\Lambda\wr\Gamma)$. This is the idea in \cite{Oz04} and a general framework is given in \cite[Section 15]{BO08}.

In \cite[Section 3]{HIK20}, we have developed a way of applying these techniques to creation operators acting on anti-symmetric Fock spaces. In this note, we introduce a more general framework that covers 
\begin{itemize}
	\item (Subsection \ref{Examples from groups}) wreath product groups and $\Z^2\rtimes \mathrm{SL}(2,\Z)$ \cite{Oz04,BO08,Oz08};

	\item (Section \ref{Bi-exactness of creation operators on Fock spaces}) creation operators on full, symmetric, and anti-symmetric Fock spaces;

	\item (Section \ref{Other examples}) free wreath product groups.

\end{itemize}
We will obtain some boundary amenability phenomenon for all these examples. 
This framework unifies Ozawa's works and our previous one, hence it is useful to understand how they are related.

As in \cite{HIK20}, our boundary amenability has an application to rigidity of associated von Neumann algebras. Recall that a discrete group action $\Gamma \act M$ on a diffuse vou Neumann algebra is \textit{solid} if for any diffuse von Neumann subalgebra $A\subset M$, which is a range of a faithful normal conditional expectation, the relative commutant $A' \cap M$ is amenable (see \cite[Appendix]{HIK20}). In Section \ref{Application to rigidity of von Neumann algebras}, we will show that, under certain assumptions, the group action $\Gamma \act X$ on a set $X$ gives rise to solid actions of $\Gamma$ on associated von Neumann algebras. Most of them were already obtained in \textit{Popa's deformation/rigidity theory}, but some of them give new examples, see Remark \ref{solid-remark}. 
In this note, we focus on the case of state preserving actions, so technical difficulties discussed in \cite[Section 4]{HIK20} do not appear.

\tableofcontents

\section{Preliminaries}
\label{Preliminaries}

\subsection{Relative $c_0$-functions and compact operators}\label{Relative functions and compact operators}

Let $X$ be a set. We recall generalized $c_0$-functions and compact operators.

\begin{df}\upshape
	Let $\mathcal X$ be a family of subsets in $X$ satisfying
	$$ E,F\in \mathcal X \quad \Rightarrow \quad E\cup F\in \mathcal X. $$
We say that a subset $A\subset X$ is \textit{small relative to $\mathcal X$}  (and write \textit{small$/\mathcal X$}) if there is $E \in \mathcal X$ such that $A\subset E$. We use the following notations.
\begin{itemize}
	\item For any net $(x_\lambda)_\lambda$ in $X$, we write $X \ni x_\lambda \to \infty /\mathcal X$ if for any $E\in \mathcal X$, there is $\lambda_0$ such that $x_\lambda\not\in E$ for all  $\lambda \geq \lambda_0$. 

	\item We denote by $c_0(\mathcal X)\subset \ell^\infty(X)$ the C$^*$-algebra generated by functions whose supports are small relative to $\mathcal X$.

\end{itemize}
\end{df}

We note that if $\mathcal X$ is the family of all finite subsets, then $c_0(\mathcal X)=c_0(X)$. The following lemma is straightforward.

\begin{lem}
	The following statements hold true.
\begin{enumerate}\label{c0-lemma}
	\item Let $c_0^{\rm alg}(\mathcal X)\subset \ell^\infty(X)$ be the set of all functions whose supports are small relative to $\mathcal X$. Then it is a dense $\ast$-subalgebra in $c_0(\mathcal X)$. In particular $c_0(\mathcal X)\leq \ell^\infty(X)$ is a closed ideal.

	\item For any $f\in \ell^\infty(X)$, the following conditions are equivalent:
\begin{itemize}
	\item[$\rm(a)$] $ f\in c_0(\mathcal X)$;

	\item[$\rm(b)$] $\{x\in X\mid |f(x)|>\varepsilon\}$ is small$/\mathcal X$ for any $\varepsilon>0$;

	\item[$\rm(c)$] for any net $(x_\lambda)_\lambda$ in $X$, $x_\lambda \to \infty /\mathcal X$ implies $\lim_{\lambda\to \infty} f(x_\lambda) = 0$ (we write it as $\lim_{x \to \infty /\mathcal X}f(x)=0$).
\end{itemize}

	\item Let $Y\subset X$ be a non-empty subset. Define 
	$$ \mathcal Y:= \{ E\cap Y\mid E\in \mathcal X \} $$
(so that $E,F\in \mathcal Y \Rightarrow E\cup F\in \mathcal Y$). We have
\begin{itemize}
	\item for any net $\{y_i\}_i$ in $Y$, \quad
	$ y_i \to \infty/ \mathcal Y \quad \Leftrightarrow \quad  y_i \to \infty/ \mathcal X$;

	\item $ c_0(\mathcal Y) = c_0(\mathcal X) \cap \ell^\infty(Y) $.

\end{itemize}

	\item Let $\Gamma \act^\pi X$ be any action of a discrete group $\Gamma $ which globally preserves $Y$ such that 
	$$\pi_g(E)\in \mathcal X,\quad \text{for all } E\in \mathcal X, \ g\in \Gamma .$$
In this case, the natural action $\Gamma \act \ell^\infty(Y)$ induces an action on $c_0(\mathcal Y)$.

\end{enumerate}
\end{lem}

Using $c_0(\mathcal X)$, define the relative compact operators and its multiplier algebra by
\begin{align*}
	\K(\mathcal X)&:= \overline{ c_0(\mathcal X)\B(\ell^2(X)) c_0(\mathcal X)}^{\rm norm}\\
	\mathrm{M}(\K(\mathcal X))&:=\text{the multiplier algebra of $\K(\mathcal X)$}.
\end{align*}
Note that $\K(\mathcal X)$ is the minimum hereditary C$^*$-algebra which contains $c_0(\mathcal X)$. If there is a discrete group action $\Gamma \act^\pi X$, by using the natural unitary representation $\Gamma \act^{U^\pi} \ell^2(X)$, $\Ad(U^\pi_g)$ globally preserves $c_0(\mathcal X)$ and $\K(\mathcal X)$, so that $ U_g^\pi \in \mathrm M(\K(\mathcal X))$.

We record the following elementary lemma.

\begin{lem}\label{lem-multiplier}
For any $a\in \B(\ell^2(X))$,
	$$ a\in \mathrm{M}(\K(\mathcal X)) \quad \Leftrightarrow \quad [a,f]\in \K(\mathcal X),\quad \text{for all } f\in c_0(\mathcal X).$$
\end{lem}

\subsection{Bi-exactness for groups}
\label{Bi-exactness for groups}

We refer the reader to \cite[Section 2]{BO08} for nuclearlity, \cite[Subsection 4.4]{BO08} for amenability of actions, \cite[Section 5]{BO08} for exactness, and \cite[Section 15]{BO08} for bi-exactness of groups.

Let $ \Gamma \act X$ be an action of a discrete group on a compact Hausdorff space. Consider the associated action $\Gamma \act C(X)$. 
We say that it is \textit{(topologically) amenable} if the reduced and full crossed products coincide: $C(X)\rtimes_{\rm red} \Gamma = C(X)\rtimes_{\rm full} \Gamma $. 
More generally for any action $ \Gamma \act A$ on a unital C$^*$-algebra $A$, the action is \textit{amenable} if the restriction to the center of $A$ is amenable. In that case, we also have $A\rtimes_{\rm red} \Gamma = A\rtimes_{\rm full} \Gamma $.

Let $\Gamma$ be a discrete group and $\cG$ a family of subsets in $\Gamma$ such that 
	$$sEt,\ E\cup F\in \cG\quad  \text{for all }s,t\in \Gamma, \ E,F\in \cG.$$
We can define $c_0(\cG)$ as in the previous subsection, which admits the left and right translation action $\Gamma\times \Gamma\act c_0(\cG)$. 
We say that $\Gamma$ is \textit{bi-exact relative to $\cG$} (and write \textit{bi-exact$/\cG$}) if the left and right translation action $\Gamma \times \Gamma \act \ell^\infty(\Gamma)/c_0(\cG)$ is amenable. When $\cG$ consists of all finite subsets in $\Gamma$, this means amenability of $\Gamma \times \Gamma \act \ell^\infty(\Gamma)/c_0(\Gamma)$. Recall that $\Gamma$ is exact if and only if the left translation $\Gamma \act \ell^\infty(\Gamma)/c_0(\Gamma)$ is amenable \cite[Theorem 5.1.7]{BO08}, hence the terminology \textit{bi}-exact makes sense.

Assume that there is a family $\cG_0$ of subgroups in $\Gamma$ which generates $\cG$ in the following sense:
	$$\cG= \{ \bigcup_{\rm finite}s\Lambda t\mid s,t\in \Gamma,\ \Lambda\in \cG_0 \}. $$
In this case, $c_0(\cG)$ coincides with $c_0(\Gamma;\cG_0)$ given in \cite[Subsection 15.1]{BO08} and hence our definition of bi-exactness is a generalization of Ozawa's one, see \cite[Proposition 15.2.3]{BO08}. The proof of 3 $\Rightarrow$ 1 in \cite[Proposition 15.2]{BO08} shows that if $\Gamma $ is bi-exact$/\cG$, then the algebraic $\ast$-homomorphism
	$$ \nu\colon C_\lambda^*(\Gamma)\ota C_\rho^*(\Gamma)\to \mathrm{M}(\K(\cG))/\K(\cG);\quad a\otimes b\mapsto ab+\K(\cG) $$
is bounded with respect to the minimal tensor norm (say, \textit{min-bounded}). Further if $\Gamma $ is countable, it has a ucp lift $\theta\colon C_\lambda^*(\Gamma)\ota C_\rho^*(\Gamma)\to \mathrm{M}(\K(\cG))$ in the sense that $\theta(a\otimes b)-ab\in \K(\cG)$ for all $a\otimes b\in C_\lambda^*(\Gamma)\otm C_\rho^*(\Gamma)$. 
This boundedness is called \textit{condition (AO)} \cite{Oz04}. In this article, we study condition (AO) for several concrete examples.

Recall that a ucp map $\pi\colon A\to B$ between unital C$^*$-algebras is nuclear if and only if for any C$^*$-algebra $C$, the map $\pi\otimes\id_C\colon A\otimes_{\rm max}C \to B\otimes_{\rm max} C$ factors through $A\otm C$. 
We will need the following facts.

\begin{prop}\label{prop-nuclear}
Let $\pi\colon A\to B$ be a unital $\ast$-homomorphism between unital C$^*$-algebras $A$ and $B$. The following assertions hold true.
\begin{enumerate}

	\item Assume that a discrete group $\Gamma$ acts on $A$ and $B$. If $\pi$ is $\Gamma$-equivariant, nuclear, and $\Gamma \actson A$ is amenable, then the natural extension $\pi\colon A\rtimes_{\rm red}\Gamma =A\rtimes_{\rm full}\Gamma \to B\rtimes_{\rm full}\Gamma$ is nuclear.

	\item If $\pi$ is nuclear and $A$ is exact, then the induced map $A/\ker\pi\to B$ is nuclear.

	\item Let $J\subset A$ and $I\subset B$ be ideals such that $\pi(I)\subset J$. Consider the following diagram
\[
\xymatrix{
0 \ar[r] &  J  \ar[r]& A \ar[r] & A/J  \ar[r] & 0  \\
0 \ar[r] & I \ar[r]\ar[u]^{\varphi=\pi|_I} &B \ar[r]\ar[u]^{\pi} & B/I \ar[u]^{\psi} \ar[r] & 0  
}
\] 
where $\psi$ is the induced map. Assume that there is an approximate unit $(a_i)_i$ in $I$ such that $(\varphi(a_i))_i$ is an approximate unit for $J$. If $\varphi$ and $\psi$ are nuclear and $B$ is exact, then $\pi$ is nuclear.

\end{enumerate}
\end{prop}
\begin{proof}
	1. This follows by the characterization of nuclearlity above.

	2. Observe first that $A$ is locally reflexive (e.g.\ \cite[Corollary 9.4.1]{BO08}). Let $C$ be any C$^*$-algebra and consider $\pi\otimes \id_C\colon A\otm C \to B\otimes_{\rm max} C$. 
With $J:=\ker\pi$ and by \cite[Corollary 9.1.5]{BO08}, it induces
	$$A/J \otm C\simeq  \frac{A\otm C}{J\otm C} \to B\otimes_{\rm max} C.$$
This shows that $A/J \otimes_{\rm max} C\to B\otimes_{\rm max} C$ factors the minimal tensor product.

	3. Observe that $\varphi^{**},\psi^{**}$ are weakly nuclear, since $\varphi,\psi$ are nuclear and $I,B/I$ are exact (hence locally reflexive). Consider the map
\begin{align*}
	I^{**}\oplus (B/I)^{**}\simeq B^{**}\to^{\pi^{**}} A^{**}\simeq J^{**}\oplus (A/J)^{**}
\end{align*}
Then $\pi^{**}$ restricts to $\varphi^{**}\colon I^{**}\to J^{**}$, which  is \textit{unital} by the assumption on approximate units. Then $\pi^{**}$ restricts to $ (B/I)^{**}\to  (A/J)^{**}$, which coincides with $\psi^{**}$ by construction. We conclude that $\pi^{**}$ is weakly nuclear. This implies $\pi$ is nuclear. Note that we can directly prove it by diagram chasing (without von Neumann algebras).
\end{proof}

\subsection{Standard representations}
\label{Standard representations}

For Tomita--Takesaki's modular theory, we refer the reader to \cite{Ta03}. For a von Neumann algebra $M$ with a faithful normal state $\varphi$, we denote by $\Delta_\varphi$ and $J_\varphi$ the \textit{modular operator} and the \textit{modular conjugation}. When $\varphi$ is tracial, $\Delta_\varphi$ is trivial and $J_\varphi$ is given by the involution $a\mapsto a^*$. The GNS representation $L^2(M):=L^2(M,\varphi)$ is called the \textit{standard representation}. 

Let $\alpha\colon \Gamma\act M$ be an action of a discrete group $\Gamma$ and $U_g \in \mathcal U(L^2(M))$ the standard implementation of $\alpha_g$ for each $g\in \Gamma$. 
We use the standard representation $L^2(M)\otimes \ell^2(\Gamma)$ of $M\rtimes \Gamma$ given by: with the $J$-map $J_M$ on $L^2(M)$,
\begin{align*}
	\text{Left action}\quad &M\ni a\mapsto a\otimes 1_\Gamma;\qquad \Gamma \ni g \mapsto U_g \otimes \lambda_g;\\
	\text{Right action}\quad &J_MMJ_M \ni a \mapsto \pi_r (a)=\sum_{h\in \Gamma}U_h a U_{h}^{-1}\otimes e_{h,h};\quad \Gamma \ni g \mapsto 1\otimes \rho_g;\\
	\text{$J$-map}\quad 	&J= \sum_{h\in \Gamma}U_hJ_M \otimes e_{h,h^{-1}}.
\end{align*}
Here $\{e_{g,h}\}_{g,h\in \Gamma}$ is the matrix unit in $\B(\ell^2(\Gamma))$. 
In the case that $M\rtimes \Gamma = L(\Delta\rtimes \Gamma)$ (a group von Neumann algebra of a semidirect product group $\Delta\rtimes \Gamma$), $\Ad(U_g\otimes \lambda_g\rho_h)$ for $g,h\in \Gamma$ restricts to an automorphism on $\ell^\infty(\Delta\rtimes \Gamma)$, which corresponds to translations $s\mapsto g^{-1}sh$ on $\Delta\rtimes \Gamma$. Similarly, $\Ad(J)$ corresponds to the inverse map on $\Delta\rtimes \Gamma$.

\subsection{Fock spaces and associated von Neumann algebras}
\label{Fock spaces and associated von Neumann algebras}

Let $q\in \{0,1,-1\}$ and let $H$ be a Hilbert space. For each $n\in \N$, we denote $H^{\ota n}$ by the $n$-algebraic tensor product. Define a (possibly degenerate) inner product on $H^{\ota n}$ by
\[
\langle \xi_1 \otimes \cdots \otimes \xi_n, \eta_1\otimes \cdots \otimes \eta_n \rangle = \sum_{\sigma \in \frakS_n}q^{i(\sigma)} \langle{\xi_{\sigma(1)}\otimes \cdots \otimes \xi_{\sigma(n)}, \eta_1 \otimes \cdots \otimes \eta_n}\rangle,
\]
for $\xi_1,\dots,\xi_n,\eta_1,\dots,\eta_n \in H$, where $i(\sigma)$ is the number of inversions. We define the \textit{$q$-Fock space} $\cF_q(H)$ by separation and completion of the algebraic Fock space
	$$\cF_{\rm alg}(H) =\C \Omega \oplus \bigoplus_{n \geq 1} H^{\ota n}.$$
When $q=0,1$, and $-1$, $\cF_q(H)$ is called the \textit{full}, \textit{symmetric}, and \textit{anti-symmetric} Fock space, respectively. For $\xi\in H$, define the \textit{left} and \textit{right creation operator} by
\begin{align*}
	& \ell(\xi)(\xi_1\otimes \cdots\otimes \xi_n) := \xi\otimes \xi_1\otimes \cdots\otimes \xi_n;
	& r(\xi)(\xi_1\otimes \cdots\otimes \xi_n) := \xi_1\otimes \cdots\otimes \xi_n\otimes \xi
\end{align*}
for all $n\in \N$ and $\xi_1,\ldots,\xi_n\in H$. They satisfy the \textit{$q$-commutation relation} 
	$$\ell(\xi)\ell(\eta)^* - q \ell(\eta)^*\ell(\xi) = \langle \xi,\eta\rangle,\quad \text{for all }\xi,\eta\in H.$$
Note that for any $0\neq \xi\in H$, $\ell(\xi)$ is bounded if $q\in\{0,-1\}$, and is unbounded if $q=1$.

We next introduce associated von Neumann algebras. For this, consider a \textit{real} Hilbert space $H_\R$ and extend it to a complex Hilbert space by $H:=H_\R\otimes_\R \C$. We identify $H_\R\otimes_\R 1 = H_\R$. Let $I$ be the involution for $H_\R \subset H$, that is, $I(\xi\otimes \lambda)=\xi\otimes \overline{\lambda}$ for $\xi\otimes \lambda\in H_\R \otimes_\R \C$. Note that $H_\R=\{ \xi\in H\mid I\xi = \xi \}$.

Assume first $q=1$. Then putting $W(\xi):= \ell(\xi)+\ell(\xi)^*$ for $\xi$, we can define the \textit{Gaussian algebra} by
	$$ \Gamma_q(H_\R):=\mathrm{W}^*\{ e^{\ri W(\xi)} \mid \xi\in H_\R\}. $$
Then $\Omega$ is a cyclic separating vector for $\Gamma_q(H_\R)$. In this case, $\Gamma_q(H_\R)$ is commutative and the vacuum state $\langle \,\cdot\, \Omega,\Omega \rangle$ corresponds the the one arising from (a product of) the Gaussian measure. See \cite[Subsection II.1]{Bo14} for these facts.

Assume next $q\in \{0,-1\}$. Let $U\colon \R\to \mathcal O(H_\R)$ be any strongly continuous representation of $\R$ on $H_\R$. We extend $U$ on $H$ by $U_t\otimes \id_{\C}$ for $t\in \R$ and use the same notation $U_t$. By Stone's theorem, take the infinitesimal generator $A$ satisfying $U_t=A^{\ri t}$ for all $t\in \R$. Consider an embedding
\begin{align*}
	&H \to H;\quad \xi \mapsto \frac{\sqrt{2}}{\sqrt{1+A^{-1}}}\, \xi =:\widehat{\xi}.
\end{align*}
We define the associated von Neumann algebra by
\begin{align*}
	& \Gamma_q(H_\R,U):= \mathrm{W}^*\{ W(\widehat{\xi}) \mid \xi \in H\},
\end{align*}
where $W(\widehat{\xi}):=\ell(\widehat{\xi}) + \ell^*(\widehat{I\xi})$. Then $\Omega$ is cyclic and separating, and the vacuum state $\langle \,\cdot\, \Omega,\Omega \rangle$ defines a faithful normal state on $\Gamma_q(H_\R,U)$. When $q=0$, $\Gamma_q(H_\R,U)$ is the \textit{free Araki--Woods factor} introduced in \cite{Sh97}. When $q=-1$ and if the action $U$ is almost periodic, $\Gamma_q(H_\R,U)$ is the classical \textit{Araki--Woods factor}, see \cite[Section 4]{HIK20}. 
For the modular theory of these von Neumann algebras, we refer the reader to \cite[Lemma 1.4]{Hi02}, which treat the case $q\in (-1,1)$ but the same proof works for $q=-1$. 

The following fact is important to us: in all cases $q\in\{1,0,-1\}$, the modular conjugation $J$ on $\cF_q(H)$ is given by
	$$ J(\xi_1\otimes \cdots\otimes \xi_n) = (I\xi_n)\otimes \cdots\otimes (I\xi_1)$$
for all $n\in \N$ and $\xi_1\ldots,\xi_n\in H$.

\section{Bi-exactness for groups acting on semigroups}\label{Bi-exactness for groups acting on semigroups}

\subsection{Examples from groups}
\label{Examples from groups}

We first review Ozawa's work on bi-exactness for some semidirect product groups.

\begin{exa}\upshape
	Let $\Lambda,\Gamma$ be (nontrivial) countable discrete groups and consider the wreath product 
	$$\Lambda\wr \Gamma = \left(\bigoplus_\Gamma \Lambda\right) \rtimes \Gamma.$$
Set $\cG := \{ \bigcup_{\rm finite}s\Gamma t\mid s,t\in \Lambda\wr \Gamma \}$. Then it is proved in \cite{Oz04} \cite[Proposition 15.3.6 and Corollary 15.3.9]{BO08} that, if $\Gamma$ is exact and $\Lambda$ is amenable, then $\Gamma$ is bi-exact$/\cG$. This means the left-right translation $\Gamma \times \Gamma \act \ell^\infty(\Gamma)/c_0(\cG)$ is amenable. When $\Gamma $ is bi-exact, then $\Lambda \wr \Gamma$ is also bi-exact.
\end{exa}
\begin{exa}\upshape
	Let $\Gamma := \mathrm{SL}(2,\Z)$ and $\Lambda := \Z^2$ and consider $\Lambda \rtimes \Gamma$ arising from the natural action $\Gamma \act \Lambda$. Set $\cG := \{ \bigcup_{\rm finite}s\Gamma t\mid s,t\in \Lambda\rtimes \Gamma \}$. Then $\Gamma$ is  bi-exact$/\cG$. Since $\Gamma$ is bi-exact, the group $\Lambda \rtimes \Gamma$ is bi-exact. See \cite{Oz08}.
\end{exa}

In both examples, semidirect product groups $\Lambda \rtimes \Gamma$ with families $\cG := \{ \bigcup_{\rm finite}s\Gamma t\mid s,t\in \Lambda\rtimes \Gamma \}$ of subsets in $\Lambda \rtimes \Gamma$ are considered. Here we briefly recall Ozawa's proofs. Later we give detailed proofs in our more general setting.

First, we consider 
\begin{align*}
	& \ell^\infty(\overline{\cG}) :=\{ f\in \ell^\infty(\Lambda\rtimes \Gamma)\mid f(s\,\cdot\, )-f,\  f(\,\cdot\, s)-f \in c_0(\cG)\text{ for all }s\in \Lambda\}.
\end{align*}
Then the left-right actions of $\Lambda$ on $\ell^\infty(\partial \cG):=\ell^\infty(\overline{\cG})/c_0(\cG)$ is trivial, hence it  can be regarded as a boundary for $\Lambda$-actions. Next we prove that the $\Gamma\times \Gamma$-action on $\ell^\infty(\partial{\cG})$ induced by the left-right translation of $\Gamma$ is amenable. Finally since $\ell^\infty(\partial{\cG})$ commutes with $C_\lambda(\Lambda)$ and $C_\rho(\Lambda)$, we have a $\Gamma\times \Gamma$-equivariant map
	$$ \ell^\infty(\partial{\cG})\otm C_\lambda(\Lambda)\otm C_\rho(\Lambda) \to \mathrm{M}(\K(\cG))/\K(\cG).$$
By the amenability of the action, we can extend this map to the reduced crossed product. By restriction, we get that 
	$$ C_\lambda(\Lambda\rtimes \Gamma)\otm C_\rho(\Lambda\rtimes \Gamma) \to \mathrm{M}(\K(\cG))/\K(\cG)$$
is min-bounded, hence $\Lambda\rtimes \Gamma$ is bi-exact$/\cG$. Note that the following algebra was also used to prove the amenability:
\begin{align*}
	 C(\overline{\cG}) :=\{ f\in \ell^\infty(\overline{\cG})\mid f(\,\cdot\, h)-f \in c_0(\cG)\text{ for all }h\in \Gamma\}.
\end{align*}
Note that in the above argument, we need the following objects.
\begin{itemize}
	\item We denote the inverse map $s\mapsto s^{-1}$ in $\Lambda\rtimes \Gamma$ by $I$. We have an anti-linear isometry $J$ on $\ell^2(\Lambda\rtimes \Gamma)$ given by $J\delta_s=\delta_{s^{-1}}$ for all $s\in \Lambda\rtimes \Gamma$. This restricts to the one $J_\Lambda$ on $\ell^2(\Lambda)$.

	\item If we write the action $\Gamma \act \Lambda$ by $\pi$, it naturally extends to a unitary representation $U^\pi\colon \Gamma \act \ell^2(\Lambda)$.

	\item Left and right regular representations $\lambda$ and $\rho$ of $\Lambda$ satisfy $\lambda_g\delta_s = \delta_{gs}$ and $\rho_g\delta_s = \delta_{sg^{-1}}$ for all $s,g\in \Lambda$.

	\item For all $g,h\in \Gamma$, $s,t\in \Lambda$, we have that (inside $\B(\ell^2(\Lambda))$)
\begin{align*}
	U^\pi_g \lambda_s U_g^{\pi *} = \lambda_{\pi_g(s)} ,\quad U^\pi_g \rho(s) U_g^{\pi *} = \rho_{\pi_g(s)},\quad 
	J_\Lambda \lambda_s J_\Lambda = \rho_{s}.
\end{align*}
\end{itemize}
We thus have $I,J_\Lambda,U_g^\pi,\lambda,\rho$ in this setting. 
Our aim of this section is to generalize the above framework to the ones for groups $\Gamma$ acting on \textit{semigroups} $S$ rather than $\Lambda$.

\subsection{Basic setting}\label{Basic setting}

	Let $S$ be a semigoup with unit $1_S$. We always assume the following two conditions:
\begin{itemize}
	\item (cancellative)  for any $x,y,s\in S$, $sx=sy$ or $xs=ys$ implies $x=y$;
	\item (involution) there is a bijection $I\colon S\to S$ such that 
\begin{align*}
I^2=\id_S \quad\text{and}\quad I(st)=I(t)I(s)\quad (s,t\in S).
\end{align*}

\end{itemize}
They are trivially satisfied when $S$ is a group. Let $\Gamma $ be a discrete group and consider an action $\Gamma\act^\pi S$ such that each $\pi_g$ for $g\in \Gamma$ preserves the structure of the involutive semigroup:
	$$\pi_g (ab) = \pi_g(a)\pi_g(b)\quad \text{for }a,b\in S\quad \text{and}\quad I\pi_g=\pi_gI.$$
Consider the semidirect product $S\rtimes \Gamma$ and we use the notation $(s,g)=sg=g \pi_g^{-1}(s)$ for $s\in S$ and $g\in \Gamma$ as an element in $S\rtimes \Gamma$. Observe that $S\rtimes \Gamma$ is also a cancellative semigroup with involution given by 
	$$I(s,g) := (\pi_g^{-1}(Is),g^{-1} ) =g^{-1}(Is).$$
Since $S\rtimes \Gamma=S\times \Gamma$ as a set, we have a natural inclusion
	$$\ell^\infty(S\rtimes \Gamma)=\ell^\infty(S)\ovt \ell^\infty(\Gamma) \subset \B(\ell^2(S)\otimes \ell^2(\Gamma)).$$
In this section, we always keep the following assumptions. They are objects in $\B(\ell^2(S))$ which correspond to structures of $S\rtimes \Gamma$ \textit{up to} multiples of $\T$ (or $\C$). Below we use the notation 
	$$a=_\C b\quad  \text{if there is $\lambda\in \C\setminus \{0\}$ such that }a=\lambda b,$$
and $a=_\T b$ if $a=\lambda b$ for $\lambda\in \T$.

\begin{assum}\label{assumption}\upshape
	Let $\Gamma\act^\pi S$ be as above. Assume the following conditions.
\begin{itemize}
	\item (Involution) There is an anti-linear isometry $J_S$ on $\ell^2(S)$ which satisfies 
	$$J_S \delta_x =_\T \delta_{Ix}\quad\text{for all }x\in S.$$

	\item (Unitary representation) There is a unitary representation $U^\pi\colon \Gamma \act \ell^2(S)$ such that
	$$ U^\pi_g\delta_{x} =_\T \delta_{\pi_g(x)}\quad\text{and}\quad U_g^\pi J_S= J_S U_g^\pi\quad  \text{for all }g\in \Gamma ,\ x\in S.$$

	\item (Creation operators) For any $s\in S$, there are $\ell(s),r(s)\in \B(\ell^2(S))$ such that for any $x\in X$,
\begin{align*}
	\ell(s)\delta_x=C_{s,x}\delta_{sx}\quad \text{and}\quad
r(s)\delta_x=C^r_{s,x}\delta_{xs}\quad \text{for some }C_{s,x},C_{s,x}^r\in \C,
\end{align*}
where $C_{s,x},C_{s,x}^r$ are possibly zero. We further assume
\begin{align*}
	\ell(s)\ell(t)=_\C \ell(st) \quad\text{and}\quad r(s)r(t)=_\C r(ts)\quad \text{for all }s,t\in S.
\end{align*}

	\item (Covariance conditions)	For all $g,h\in \Gamma$, $s,t\in S$, we have
\begin{align*}
	&U^\pi_g \ell(s) U_g^{\pi *}= \ell(\pi_g(s)) ,\quad U^\pi_g r(s) U_g^{\pi *} = r(\pi_g(s));\\
	&J_S \ell(s) J_S = r(Is) ,\quad J_S r(s) J_S = \ell(Is).
\end{align*}

	\item (Relativity) There is a family $\cG$ of subsets in $S\rtimes \Gamma$ such that
\begin{align*}
	E\cup F \in \cG, \quad  sEt \in \cG, \quad \text{and}\quad IE \in \cG,\quad \text{for all }E,F\in \cG, \ s,t\in S\rtimes \Gamma.
\end{align*}
Assume $\Gamma \not \in \cG$, so that $c_0(\cG)\subsetneq \ell^\infty(S\rtimes \Gamma)$.
\end{itemize}
\end{assum}

Observe that examples from groups in the previous subsection trivially satisfy all the assumptions (regular representations of $\Lambda$ correspond to creations). 
Our main examples are ones from creations operators on Fock spaces, see Subsection \ref{Fock spaces as semigroups}. 

Using the family $\cG$, we can define relative algebras:
\begin{align*}
	& c_0(\cG):=\{f\in \ell^\infty(S\rtimes \Gamma)\mid \lim_{x\to \infty/\cG}f(x)=0\};\\
	&\K(\cG):= \overline{ c_0(\cG)\B(\ell^2(S\rtimes \Gamma)) c_0(\cG)}^{\rm norm}.
\end{align*}
We denote by $\mathrm{M}(\K(\cG))$ the multiplier algebra of $\K(\cG)$. Consider the following representations, 
\begin{align*}
	\text{Left action}\quad &\pi_\ell\colon \B(\ell^2(S))\ni a\mapsto a\otimes 1_\Gamma;\qquad \Gamma \ni g \mapsto U_g \otimes \lambda_g;\\
	\text{Right action}\quad &\pi_r\colon \B(\ell^2(S)) \ni a \mapsto \sum_{h\in \Gamma}U_h^\pi  a  U_{h}^{\pi*}\otimes e_{h,h};\quad \Gamma \ni g \mapsto 1\otimes \rho_g;\\
	\text{$J$-map}\quad & J=\sum_{h\in \Gamma} U_h^\pi J_S\otimes e_{h,h^{-1}}.
\end{align*}
Note that this coincides with the one in Subsection \ref{Standard representations} when $S$ is a group. 

\begin{rem}\upshape
	We assumed that $\ell(s),r(s)$ are all bounded operators. In the case of symmetric Fock spaces, however, creation operators are all \textit{unbounded}. We will prove some bi-exactness results for this case, so we sometimes consider below the case that $\ell(x),r(x)$ are closed operators. This exactly means that if we regard constants $C_{s,x},C_{s,x}^r$ as functions on $x\in S$ denoted by $C_{s,\cdot}$ and $C^r_{s,\cdot}$, then they are (possibly) unbounded functions.

We also note that, if $a=\ell(s)$ for $s\in S$ is a closed operator, then $\pi_\ell(a)=a\otimes 1$ and $\pi_r(a)=J (J_SaJ_S\otimes 1) J$ are again closed, hence $\pi_\ell,\pi_r$ are defined on closed operators on $\ell^2(S)$. In particular $\pi_\ell(C_{s,\cdot})$ and $\pi_r(C^r_{s,\cdot})$ define closed operators.
\end{rem}

Before proceeding, we see two elementary lemmas. We omit most proofs.

\begin{lem}\label{lem-action}
The following conditions hold true.
\begin{enumerate}
	\item For all $g\in \Gamma$ and $a\in \B(\ell^2(S))$, we have
	$$J(U_g^\pi\otimes \lambda_g)J = 1\otimes \rho_g,\quad \pi_r(a)=J (J_SaJ_S\otimes 1) J.$$

	\item $\Ad(U_g\otimes \lambda_g)$, $\Ad(1\otimes \rho_g)$ for $g\in \Gamma$, and $\Ad J$ restrict to automorphisms on $\ell^\infty(S\rtimes \Gamma)$, which correspond to left, right translations by $g$ and the involution $I$. More precisely, 
\begin{align*}
	& \Ad(U_g^\pi\otimes \lambda_g\rho_h)f = f(g^{-1}\, \cdot \, h),\quad \text{for }f\in \ell^\infty(S\rtimes \Gamma);\\
	& \Ad(J)f = \overline{f\circ I} ,\qquad \text{for }f\in \ell^\infty(S\rtimes \Gamma). 
\end{align*}
It follows that
\begin{itemize}
	\item $\Ad(U_g^\pi\otimes \lambda_g\rho_h)$ and $\Ad(J)$ globally preserve $c_0(\cG)$ and $\K(\cG)$;

	\item for any (possibly unbounded) map $f\colon S\to \C$,
	$$\pi_\ell(f)=f\otimes 1,\quad \pi_r(f)= (f\circ I_S\otimes 1)\circ I$$
can be regarded as maps on $S\rtimes \Gamma$.

\end{itemize}
\end{enumerate}
\end{lem}

\begin{lem}\label{lem-basic-semigp}
	Fix $s\in S$ and assume that $\ell(s),r(s)$ are possibly unbounded. The following conditions hold true.
\begin{enumerate}
	\item For any map $f\in\ell^\infty(S\rtimes \Gamma)$, we have
\begin{align*}
	& f\pi_\ell(\ell(s))= \pi_\ell(\ell(s))f(s\,\cdot \, ) ,\quad \pi_\ell(\ell(s))f=f(s^{-1}\, \cdot \, )1_{s (S\rtimes \Gamma)}\pi_\ell(\ell(s));\\
	&f\pi_r(r(s))= \pi_r(r(s)) f(\cdot \, s \, ),\quad  \pi_r(r(s)) f=f(\, \cdot \, s^{-1})1_{ (S\rtimes \Gamma)s}\pi_r(r(s)).
\end{align*}
where $1_X$ means the characteristic function on a subset $X\subset S\rtimes \Gamma$.

	\item For any $f\in\ell^\infty(S\rtimes \Gamma)$, the following implications hold:
\begin{align*}
	&\pi_\ell(C_{s,\cdot})(f-f(s\,\cdot \, ))\in c_0(\cG)\quad \Rightarrow \quad [\pi_\ell(\ell(s)) , f]\in \K(\cG);\\
	&\pi_r(C^r_{s,\cdot})(f-f(\,\cdot \, s))\in c_0(\cG)\quad \Rightarrow \quad [\pi_r(r(s)) , f]\in \K(\cG).
\end{align*}
In particular $\pi_\ell(\ell(s)),\pi_r(r(s))\in \mathrm{M}(\K(\cG))$ if $\ell(s),r(s)$ are bounded.
\end{enumerate}

\end{lem}
\begin{proof}
	1. Apply them to $\delta_x\otimes \delta_s$ for $x\in S$ and $g\in \Gamma$.

	2. We see the case $\pi_\ell$. Observe that $\ell(s)C_{s,\cdot}^{-1}$ is a well defined isometry. Then item 1 implies 
\begin{align*}
	[\pi_\ell(\ell(s)) , f] 
	= \pi_\ell(\ell(s))(f-f(s\,\cdot \, ) = \pi_\ell(\ell(s)C_{s,\cdot}^{-1}) \pi_\ell(C_{s,\cdot})(f-f(s\,\cdot \, ) ).
\end{align*}
Since $\pi_\ell(\ell(s)C_{s,\cdot}^{-1})$ is an isometry, if $ \pi_\ell(C_{s,\cdot})(f-f(s\,\cdot \, ) )\in c_0(\cG)$, we have $[\pi_\ell(\ell(s)) , f]\in \B(\ell^2(S\rtimes \Gamma))c_0(\cG)$. A similar argument works for $ [\pi_\ell(\ell(s)) , f] = (f(s^{-1}\, \cdot \, )1_{s(S\rtimes \Gamma)}-f)\pi_\ell(\ell(s))$, hence we get the conclusion. The last statement follows by Lemma \ref{lem-multiplier}.
\end{proof}

\subsection*{Boundary C$^*$-algebras and associated actions}

Keep Assumption \ref{assumption} and define C$^*$-algebras by
\begin{align*}
	&\ell^\infty(\overline{\cG}):=\{ f\in \ell^\infty(S\rtimes \Gamma)\mid \pi_\ell(C_{s,\cdot})(f-f(s\,\cdot \, )),\pi_r(C^r_{s,\cdot})(f-f(\,\cdot \, s))\in c_0(\cG)\ \text{for all }s\in S\};\\
	&C(\overline{\cG}):=\{ f\in \ell^\infty(\overline{\cG})\mid f-f(\,\cdot \, h)\in c_0(\cG)\ \text{for all }h\in \Gamma\}.
\end{align*}
Here $\ell^\infty(\overline{\cG})$ should be regarded as a boundary for left-right creations (and $C(\overline{\cG})$ for the right $\Gamma$-action). Since they contain $c_0(\cG)$, we can define the following \textit{boundary} C$^*$-algebras
\begin{align*}
	 \ell^\infty(\partial \cG):=\ell^\infty(\overline{\cG})/c_0(\cG), \quad
	  C(\partial \cG):=C(\overline{\cG})/c_0(\cG).
\end{align*}

\begin{lem}\label{lem-boundary-action}
The map $\Ad(J)$ globally preserves $\ell^\infty(\overline{\cG})$. For any $g\in \Gamma$, $\Ad(U_g^\pi\otimes \lambda_g)$ and $\Ad(1\otimes \rho_g)$ globally preserves $ \ell^\infty(\overline{\cG})$ and $C(\overline{\cG})$.
\end{lem}
\begin{proof}
	Let $f\in \ell^\infty(\overline{\cG})$ and put $F:=JfJ=\overline{f\circ I}$. Then 
\begin{align*}
	 \pi_r(C_{s,I\cdot})(F - F ( \,\cdot \, (Is)))= J\pi_\ell(C_{s,\cdot})(f-f(s\,\cdot \, ))J\in Jc_0(\cG)J=c_0(\cG). 
\end{align*}
Since $C_{s,I\cdot}=\overline{C^r_{Is,\cdot}}$ (which comes from $\ell(s)=J_Sr(Is)J_S$), by putting $Is=:t$, we have
	$$\pi_r(\overline{C^r_{t,\cdot}})(F - F ( \,\cdot \, t))\in c_0(\cG).$$
This means $F$ satisfies the condition for $\pi_r$. Similarly $F$ satisfies the desired condition for $\pi_\ell$, hence $F\in \ell^\infty(\overline{\cG})$. For group actions of $\Gamma$, we can use a similar argument (use the equation $U_g^\pi \ell(s)U_g^{\pi *} = \ell(\pi_g(s))$).
\end{proof}

By the lemma, we get group actions arising from adjoint maps: for $g,h\in \Gamma$,
\begin{align*}
	\Ad(U_g^\pi\otimes \lambda_g\rho_h)&\colon \Gamma\times \Gamma \act \ell^\infty(\partial \cG);\\
	\Ad(U_g^\pi\otimes \lambda_g)&\colon \Gamma \act C(\partial \cG).
\end{align*}
By Lemma \ref{lem-action}, they coincide with the ones arising from translation actions on $S\rtimes \Gamma$.

\subsection{Bi-exactness and condition (AO)}
\label{Bi-exactness and condition (AO)}

We now define bi-exactness in our setting. We keep $S,\Gamma,\cG$ in assumption \ref{assumption}.

\begin{df}\upshape
	We say that $\Gamma \act^\pi S$ is \textit{bi-exact relative to $\cG$} (or \textit{bi-exact/$\cG$} in short) if the $\Gamma$-action $\Gamma \times \Gamma \act \ell^\infty(\partial\cG)$ arising from the left and right translations is (topologically) amenable.
\end{df}
\begin{rem}\upshape
	As we will explain at the last part in this subsection, for certain examples, we should exchange $S$ with a subset $S_0\subset S$.
\end{rem}

Note that this notion obviously depends on other objects such as $J_S,U^\pi,\ell,r$. Note that $\Gamma$ is automatically exact since it has a (topologically) amenable action.

Here is a sufficient condition of bi-exactness. We will indeed check this condition later.

\begin{lem}\label{bi-exact-lem}
	If the left translation $\Gamma \act C(\partial{\cG})$ is amenable, then $\Gamma \act S$ is bi-exact/$\cG$.
\end{lem}
\begin{proof}
	Consider the C$^*$-algebra 
\begin{align*}
	&JC(\overline{\cG})J=\{ f\in \ell^\infty(\overline{\cG})\mid f-f(g \,\cdot \, )\in c_0(\cG)\ \text{for all }g\in \Gamma\}.
\end{align*}
This is a boundary for the left $\Gamma$-action. Since $\Ad(J)$ globally preserves $c_0(\cG)$, it induces an anti-linear $\ast$-isomorphism on $\mathrm{M}(\K(\cG))$, which we again denote by $\Ad(J)$. Then it follows that
	$$\Ad(J)(C(\partial{\cG})) = JC(\overline{\cG})J/c_0(\cG) =:JC(\partial \cG)J.$$
By applying $\Ad(J)$ to $\Gamma \act C(\partial{\cG})$, the right translation $\Gamma \act JC(\partial \cG)J$ is also amenable. Observe that there are inclusions
	$$C(\partial{\cG}),JC(\partial{\cG})J\subset \ell^\infty(\partial{\cG}).$$
Then the left translation on $\ell^\infty(\partial{\cG})$ restricts to an amenable action on $C(\partial{\cG})$, so the crossed product $ \ell^\infty(\partial{\cG})\rtimes_{\rm red} \Gamma$ is nuclear. Next consider the right translation on $\ell^\infty(\partial{\cG})$. It commutes with the left translation and is amenable on $JC(\partial{\cG})J$. Since the left translation on $JC(\partial{\cG})J$ is trivial, $JC(\partial{\cG})J$ is contained in the center of $ \ell^\infty(\partial{\cG})\rtimes_{\rm red} \Gamma$. This implies that 
	$$[\ell^\infty(\partial{\cG})\rtimes_{\rm red} \Gamma]\rtimes_{\rm red} \Gamma=  \ell^\infty(\partial{\cG})\rtimes_{\rm red} (\Gamma\times \Gamma)$$
is nuclear. This means that the $\Gamma\times \Gamma$-action is topologically amenable.
\end{proof}

We next study relationship between bi-exactness and condition (AO). Define
	$$ \cC:=\mathrm{C}^*\{\ell(s)\mid s\in S\},\quad  \cC_r:=\mathrm{C}^*\{r(s)\mid s\in S\} =J_S\cC J_S\subset \B(\ell^2(S)).$$
By the covariance conditions, $\Ad U^\pi$ defines actions on $\cC,\cC_r$. We can consider the left and right representations $\pi_\ell,\pi_r$ for crossed products of these actions, hence
	$$\pi_\ell(\cC\rtimes_{\rm red}\Gamma) , \pi_r(\cC_r\rtimes_{\rm red}\Gamma) \subset \B(\ell^2(S\rtimes \Gamma)),$$
Note that $\pi_r(\cC_r\rtimes_{\rm red}\Gamma)=J\pi_\ell(\cC\rtimes_{\rm red}\Gamma)J$. The next lemma is straightforward, see Lemmas \ref{lem-action} and \ref{lem-basic-semigp}.

\begin{lem}
We have $ \pi_\ell(\cC\rtimes_{\rm red}\Gamma),\ \pi_r(\cC_r\rtimes_{\rm red}\Gamma) \subset \mathrm{M}(\K(\cG))$.
\end{lem}

Let $B\subset \cC$ be a unital C$^*$-subalgebra such that the $\Gamma$-action on $\cC$ globally preserves $B$ and that $[B,B_r]=0$, where $B_r:=J_S BJ_S$. For example, we will choose $B=\mathrm{C}^*\{W(\xi)\mid \xi\in H\}$ in the Fock space case. By the previous lemma, one can define an algebraic $\ast$-homomorphism
\begin{align*}
	&\nu\colon (B\rtimes_{\rm red}\Gamma) \ota (B_r \rtimes_{\rm red}\Gamma) \to \mathrm{M}(\K(\cG))/\K(\cG);\qquad a\otimes b\mapsto \pi_\ell(a)\pi_r(b)+\K(\cG).
\end{align*}
Since $\ell^\infty(\overline{\cG}) \subset \mathrm{M}(\K(\cG))$ and $c_0(\cG) \subset \K(\cG)$, there is a $\ast$-homomorphism
	$$\pi_{\cG}\colon \ell^\infty(\partial{\cG}) \to \mathrm{M}(\K(\cG))/\K(\cG).$$
Then by the definition of $\ell^\infty(\overline{\cG})$, we have an algebraic $\ast$-homomorphism 
\begin{align*}
	\nu_\cG\colon  \ell^\infty(\partial{\cG}) \ota B \ota B_r  \to \mathrm{M}(\K(\cG))/\K(\cG);\quad \nu_\cG=\pi_\cG\otimes \pi_\ell\otimes \pi_r.
\end{align*}
The next proposition explains how to use our bi-exactness.  We will call the boundedness of $\nu$ below \textit{condition (AO) relative to $\K(\cG)$}.

\begin{prop}\label{bi-exact-prop}
	Keep the setting and assume that
\begin{itemize}
	\item $\nu_\cG$ is min-bounded and nuclear (e.g.\ $B$ is nuclear);
	\item $\Gamma\act S$ is bi-exact$/\cG$.
\end{itemize}
Then $\nu$ is min-bounded and nuclear. Further if $B\rtimes_{\rm red}\Gamma$ is separable, then $\nu$ has a ucp lift, that is, there is a ucp map 
	$$\theta\colon (B\rtimes_{\rm red}\Gamma) \otm (B_r \rtimes_{\rm red}\Gamma)\to \mathrm{M}(\K(\cG))$$ 
such that $\theta(a\otimes b) - \pi_\ell(a)\pi_r(b) \in \K(\cG)$ for all $a\in B\rtimes_{\rm red}\Gamma$ and $b\in B_r\rtimes_{\rm red}\Gamma$.
\end{prop}
\begin{proof}
	Put $\widetilde{\Gamma}:=\Gamma\times \Gamma$. Consider $\widetilde{\Gamma}$-actions on $\ell^\infty(\partial{\cG})$ and $B \otm B_r$ by the left and right translations and by $\Ad(U_g^\pi \otimes U_h^\pi)$ for $(g,h)\in \widetilde{\Gamma}$ respectively. 
By assumption, the tensor product $\widetilde{\Gamma}$-action on $\ell^\infty(\partial{\cG}) \otm B \otm B_r$ is amenable. Observe that $\nu_\cG$ is $\widetilde{\Gamma}$-equivariant from this action to the one given by $\Ad(U_g^\pi \otimes \lambda_g\rho_h)$ for $(g,h)\in \widetilde{\Gamma}$, hence the following map is bounded and is nuclear, see item 1 in Proposition \ref{prop-nuclear},
\begin{align*}
	[\ell^\infty(\partial{\cG}) \otm B \otm B_r]\rtimes_{\rm red}\widetilde{\Gamma} \to \mathrm{M}(\K(\cG))/\K(\cG).
\end{align*}
The domain contains
	$$[\C \otm B \otm B_r]\rtimes_{\rm red}\widetilde{\Gamma} \simeq (B\rtimes_{\rm red}\Gamma) \otm (B_r \rtimes_{\rm red}\Gamma) $$
and the restriction to this algebra coincides with $\nu$. This is the conclusion. If $B\rtimes_{\rm red}\Gamma$ is separable, we can use Choi--Effros's lifting theorem (e.g.\ \cite[Theorem C.3]{BO08}).
\end{proof}

\subsection*{A modification for anti-symmetric Fock spaces}

Assumption \ref{assumption} contains the case that $\ell(s)=0$ for some $s\in S$. This indeed happens in the case of anti-symmetric Fock space. In such a case, we should ignore non-important part in $\ell^2(S)$. 
To do this, we keep Assumption \ref{assumption} and consider 
	$$ S_0:= \{s\in S\mid \ell(s)\neq 0 \text{ or }\ell(Is)\neq 0\}=\{ s\in S\mid C_{s,\cdot}\neq 0\text{ or } C_{Is,\cdot}\neq 0 \} . $$
Observe that by the covariance conditions, $\pi$ and $I$ restrict to bijective maps on $S_0$ and that $S_0$ can be defined by right creations. Write $S_0\rtimes \Gamma := S_0\times \Gamma$ and we denote by $P_{S_0}\in \B(\ell^2(S))$ the orthogonal projection onto $\ell^2(S_0)$. We use the following items in $\B(\ell^2(S_0))$:  for $g\in \Gamma$ and $s\in S$,
	$$U^\pi_g P_{S_0},\quad J_SP_{S_0},\quad \ell(s)P_{S_0}\quad\text{and}\quad  r(s)P_{S_0}.$$
With abuse of notations, we again denote by $U_g^\pi,J_{S},\ell(s),r(s)$ respectively. As we have seen in Lemma \ref{c0-lemma}, $\cG_0:=\cG\cap (S_0\rtimes \Gamma)$ satisfies 
	$$ c_0(\cG_0)=c_0(\cG)\cap \ell^\infty(S_0\rtimes \Gamma)=c_0(\cG) (P_{S_0}\otimes 1),\quad \K(\cG_0)=(P_{S_0}\otimes 1)\K(\cG) (P_{S_0}\otimes 1).$$
Then it is straightforward to check that all lemmas in Subsection \ref{Basic setting} hold for $S_0\rtimes \Gamma$ (one can apply the compression map by $P_{S_0}\otimes 1$ to conclusions of lemmas). We get actions 
	$$\Gamma \times \Gamma \act \ell^\infty(\partial{\cG}_0)\quad \text{and}\quad\Gamma \act C(\partial \cG_0).$$
In this case, we say that $\Gamma \act^\pi S_0$ is \textit{bi-exact relative to $\cG_0$} if this $\Gamma\times \Gamma$-action is amenable. This is weaker than bi-exactness relative to $\cG$, but is enough to prove Proposition \ref{bi-exact-prop}. More precisely if we exchange $S\rtimes \Gamma$ and $\cG$ in Proposition \ref{bi-exact-prop} with $S_0\rtimes \Gamma$ and $\cG_0$, the same proof still works.

Thus we can use all the results in Subsections \ref{Basic setting} and \ref{Bi-exactness and condition (AO)} (including Lemma \ref{intersection-AO-lem} below) if we exchange $S$ with $S_0$ which supports all creation operators.

\subsection*{An intersection lemma}

We keep the setting from Proposition \ref{bi-exact-prop}. If there is another family $\cF$ of subsets in $S\rtimes \Gamma$ satisfying the relativity condition in Assumption \ref{assumption}, then 
	$$ \cG\cap \cF:=\{ G\cap F \mid G\in \cG,\ F\in \cF\} $$
also satisfy the same condition. It is then natural to ask when we get condition (AO) relative to $\K(\cG\cap \cF)$, using information of $\cG$ and $\cF$. It was discussed in \cite[Proposition 15.2.7]{BO08} for the group case.

In our setting, we need to assume that $\cF$ is of the following form: assume that there is a family $\mathcal E$ of subsets in $\Gamma$ such that 
\begin{align*}
	& E^{-1},\ gE h \in \mathcal E,\ E\cup F\quad \text{for all }E \in \mathcal E,F \text{ and }g,h\in \Gamma.
\end{align*}
Then we put $\cF:=\{ S E \mid E \in \mathcal E\}$. Since $SE=ES$ for $E\in \mathcal E$, it is easy to check the relativity condition in Assumption \ref{assumption} for $\cF$. 
In this notation, if $\Gamma$ is bi-exact$/\mathcal E$, then $\Gamma\act S$ is bi-exact$/\cF$. To see this, we have only to see $1\otimes \ell^\infty(\Gamma)\subset \ell^\infty(\overline{\cG})$, which induces a $\Gamma\times \Gamma$-equivariant map
	$$ \ell^\infty(\Gamma)/c_0(\mathcal E) \to \ell^\infty(\partial{\cG}). $$
For such $\cG$ and $\cF$, we prepare the following notation:
\begin{itemize}

	\item $\K_\cG:=\K(\cG)$ and $\rM_\cG:=\rM(\K_\cG)$ (similarly for $\cF$ and $\cG\cap \cF$);

	\item $\K:=\K_\cG\cap \K_\cF \ (=\K_{\cG\cap \cF})$ and $\mathrm{M}:=\mathrm{C}^*\{ \pi_\ell(\cC),\pi_r(\cC_r),\ell^\infty(S\rtimes \Gamma) \}$.

\end{itemize}
With abuse of notation, we omit intersections in quotients such as
	$$ \rM/\K:= \rM/(\K\cap \rM). $$
As in Proposition \ref{bi-exact-prop}, we consider $\ast$-homomorphisms (with smaller ranges)
\begin{align*}
	&\nu_{\cG}\colon \ell^\infty(\partial \cG)\ota B\ota B_r \to \rM/\K_\cG,\quad \nu_\cG = \pi_\cG\otimes \pi_\ell\otimes \pi_r; \\
	&\nu_{\cF}\colon \ell^\infty(\partial \cF)\ota B\ota B_r \to \rM/\K_\cF,\quad \nu_\cF = \pi_\cF\otimes \pi_\ell\otimes \pi_r.
\end{align*}
In this notation, we prove the following lemma. Note that this also holds for $S_0\subset S$ in the above sense.

\begin{lem}\label{intersection-AO-lem}
	Keep the setting and assume that
\begin{itemize}
	\item $\pi_\ell\otimes \pi_r\colon B\ota B_r \to \rM/\K$ is min-bounded;

	\item $\nu_\cG,\nu_\cF$ are min-bounded and nuclear;

	\item $\Gamma \act S$ is bi-exact$/\cG$ and $\Gamma$ is bi-exact$/\mathcal E$.
\end{itemize}
Then $B\rtimes_{\rm red}\Gamma$ satisfies condition $\rm (AO)$ relative to $\K_{\cG\cap \cF}$ and the associated bounded map is nuclear. It has a ucp lift if $B\rtimes_{\rm red}\Gamma $ is separable.
\end{lem}
\begin{proof}
	We put $\widetilde{\Gamma}:=\Gamma\times \Gamma$ and
\begin{align*}
	\cA:=\mathrm{C}^*\{ \pi_\ell(B),\pi_r(B_r), \ell^\infty(\overline{\cG}) \},\quad 
	\mathcal B:=\mathrm{C}^*\{ \pi_\ell(B),\pi_r(B_r), 1\otimes \ell^\infty(\Gamma) \}.
\end{align*}
Observe that $\cB\subset \cA \subset \rM$ and $\widetilde{\Gamma}$-actions on $\cB,\cA,\rM$ are defined by $\Ad(U_g^\pi\otimes \lambda_g\rho_h)$ for $(g,h)\in \widetilde{\Gamma}$, which induces ones on quotients such as $\widetilde{\Gamma}\act \rM/\K$.

\begin{claim}
	Assume that the $\ast$-homomorphism 
	$$\theta\colon [\cB/\K] \rtimes_{\rm alg}\widetilde{\Gamma}\to [\mathrm{M}/\K ] \rtimes_{\rm alg}\widetilde{\Gamma}$$
arising from the inclusion $\cB\subset \rM$ is bounded from the reduced norm to the full norm and is nuclear. Then the conclusion follows.
\end{claim}
\begin{proof}
	We consider the following maps:
\begin{align*}
	&(B\rtimes_{\rm red}\Gamma)\otm (B_r\rtimes_{\rm red}\Gamma)\simeq [B\otm B_r ] \rtimes_{\rm red}\widetilde{\Gamma}\to [\cB /\K]  \rtimes_{\rm red}\widetilde{\Gamma}\\
	&\quad \qquad \to^\theta [\mathrm{M}/\K] \rtimes_{\rm full}\widetilde{\Gamma} \to \rM_{\cG\cap\cF}/\K.
\end{align*}
Here the first map is from $B\otm B_r \to \cB /\K$ (in the assumption), and the last one is from the inclusion $\mathrm{M}\subset \rM_{\cG\cap\cF}$ with $\widetilde{\Gamma}\ni (g,h)\mapsto U_g^\pi\otimes \lambda_g\rho_h$. Then since $\theta$ is nuclear, the conclusion follows.
\end{proof}

To study $\theta$ in the claim, we consider the following two exact sequences:
\[
\xymatrix{
0 \ar[r] &  [(\K_{\cF}\cap \mathrm{M})/\K ]\rtimes_{\rm full}\widetilde{\Gamma}\ar[r] & [\mathrm{M}/\K]  \rtimes_{\rm full}\widetilde{\Gamma}\ar[r] &  [\mathrm{M}/\K_{\cF}] \rtimes_{\rm full}\widetilde{\Gamma} \ar[r] & 0 \\
0 \ar[r] &[ ( \K_{\cF}\cap \cA)/\K]\rtimes_{\rm red}\widetilde{\Gamma} \ar[r]\ar[u]^{\varphi} &[\cA/\K] \rtimes_{\rm red}\widetilde{\Gamma}\ar[r]\ar[u]^{\pi} & [ \cA/\K_{\cF}]\rtimes_{\rm red}\widetilde{\Gamma} \ar[u]^{\psi} \ar[r] & 0  ,
}
\] 
which arise from surjective $\widetilde{\Gamma}$-equivariant maps $\rM/\K\to \rM/\K_\cF$ and $\cA/\K\to \cA/\K_\cF$. Here $\varphi,\pi,\psi$ arise from inclusions and defined only at the algebraic level. 
Note that the bottom line is exact since $\widetilde{\Gamma}$ is exact. The first step of the proof is to show the following claim.

\begin{claim}
	The maps $\varphi,\pi,\psi$ are bounded.
\end{claim}
\begin{proof}
	Observe that $\pi_{\cF}((\ell^\infty(\Gamma)/c_0(\mathcal E))$ is contained in the center of $\mathcal A/\K_\cF$. Then $\widetilde{\Gamma} \act \cA/\K_\cF$ is amenable since it restricts to the amenable action $\widetilde{\Gamma} \act\pi_{\cF}((\ell^\infty(\Gamma)/c_0(\mathcal E))$. We get that $[\mathcal A/\K_\cF ] \rtimes_{\rm red}\widetilde{\Gamma}=\mathcal [A/\K_\cF ] \rtimes_{\rm full}\widetilde{\Gamma}$, so that $\psi$ is bounded.

We next see $\varphi$. If we put $\cI_\cF:=\K_\cF\cap \cA$, $\cI_\cG:=\K_\cG\cap \cA$, then there is an inclusion as a closed ideal
\begin{align*}
	&(\K_\cF\cap \cA) /\K=\cI_\cF /(\cI_\cG\cap \cI_\cF)\simeq (\cI_\cG+ \cI_\cF) /\cI_\cG \subset \cA/\cI_{\cG},
\end{align*}
which preserves $\widetilde{\Gamma}$-actions. 
Then $\widetilde{\Gamma} \act \cA/\cI_{\cG}$ is amenable since it restricts to the amenable action $\widetilde{\Gamma} \act \pi_\cG(\ell^\infty(\partial\cG))$ contained in the center. We get  $[\cA/\cI_{\cG}]\rtimes_{\rm red} \widetilde{\Gamma}=[\cA/\cI_{\cG}]\rtimes_{\rm full} \widetilde{\Gamma}$, hence its restriction to the closed ideal gives $[(\K_\cF\cap \cA) /\K]\rtimes_{\rm red} \widetilde{\Gamma}=[(\K_\cF\cap \cA) /\K]\rtimes_{\rm full} \widetilde{\Gamma}$. This implies the boundedness of $\varphi$.

We finally see $\pi$. Since we have already proved that the reduced and the full crossed product coincide on domains on $\psi,\varphi$, by the 5 lemma, the same holds for the domain of $\pi$. Thus $\pi$ is also bounded.
\end{proof}

We next see some nuclearlity of $\varphi$ and $\psi$.

\begin{claim}
	The maps $\varphi$ and the restriction of $\psi$ to $[\cB/\K_\cF]\rtimes_{\rm red}\widetilde{\Gamma}$ are nuclear.
\end{claim}
\begin{proof}
Observe first that by assumption and item 2 in Proposition \ref{prop-nuclear}, inclusions 
	$$\cA/\K_\cG =\mathrm{Im}(\nu_\cG ) \subset \rM/\K_\cG \quad\text{and}\quad  (\cB/\K_\cF \subset )\ \mathrm{Im}(\nu_\cF )\subset \rM/\K_\cF $$
are nuclear. Since $\widetilde{\Gamma}\act \cA/\K_\cG,\cB/\K_\cF$ are amenable, by item 1 in Proposition \ref{prop-nuclear}, 
	$$[\cA/\K_\cG]\rtimes_{\rm red}\widetilde{\Gamma}\to [\rM/\K_\cG]\rtimes_{\rm full}\widetilde{\Gamma} \quad\text{and}\quad  [\cB/\K_\cF]\rtimes_{\rm red}\widetilde{\Gamma} \to  [\rM/\K_\cF]\rtimes_{\rm full}\widetilde{\Gamma} $$
are nuclear. The second map is the restriction of $\psi$. We restricts the first map to $[(\K_\cF\cap \cA) /\K]\rtimes_{\rm red}\widetilde{\Gamma}$ as in the proof of the last claim. Then the same reasoning shows that the range $[\rM/\K_\cG]\rtimes_{\rm full}\widetilde{\Gamma}$ contains $[(\K_\cF\cap \rM) /\K]\rtimes_{\rm full}\widetilde{\Gamma}$ as a closed ideal. Then the restriction to these ideals is again nuclear and it coincides with $\varphi$.
\end{proof}

We now deduce our conclusion as follows.

\begin{claim}
	The restriction of $\pi$ to $[\cB/\K]\rtimes_{\rm red}\widetilde{\Gamma}$ is nuclear. Hence the conclusion follows by the first claim.
\end{claim}
\begin{proof}
	We restrict $\varphi,\pi,\psi$ to 
\[
\xymatrix{
0 \ar[r] & [( \K_{\cF}\cap \cB)/\K]\rtimes_{\rm red}\widetilde{\Gamma} \ar[r] &[\cB/\K ]\rtimes_{\rm red}\widetilde{\Gamma}\ar[r] &  [\cB/\K_{\cF}]\rtimes_{\rm red}\widetilde{\Gamma} \ar[r] & 0.
}
\] 
We will apply item 3 in Proposition \ref{prop-nuclear}. For this, since $\varphi,\psi$ are nuclear by the last claim and since $[\cB/\K]\rtimes_{\rm red}\widetilde{\Gamma}$ is exact (because $\cB/\K$ is an image of the exact algebra $\ell^\infty(\Gamma)\otm B\otm B_r$), we have only to check the assumption on approximate units.

To see this, since $\varphi$ arises from the inclusion $\K_{\cF}\cap \cB\subset \K_{\cF}\cap \rM$, it suffices to show that there is an approximate unit for $\K_{\cF}\cap \cB$ which is also one for $\K_{\cF}\cap \rM$. Then we can use $(1_{SE})_{E\in \mathcal E}$, which is an approximate unit for $\K_{\cF}$ and is contained in $\K_{\cF}\cap \cB$. 
\end{proof}
\end{proof}

\section{Bi-exactness of creation operators on Fock spaces}\label{Bi-exactness of creation operators on Fock spaces}

\subsection{Fock spaces as semigroups}\label{Fock spaces as semigroups}

Throughout this section, we keep the following setting and notation.

Let $\Gamma\act^\pi X$ be an action of a discrete group $\Gamma$ on a set $X$. Consider the associated unitary representation
	$$\pi\colon\Gamma \act \ell^2(X);\quad \pi_g\delta_x=\delta_{\pi_g(x)}.$$
For $q\in\{0,1,-1\}$, define the $q$-Fock space
	$$\cF_q:=\cF_q(\ell^2(X)) =\C \Omega \oplus \bigoplus_{n \geq 1}\ell^2(X)_q^{\otimes n},$$
where $\ell^2(X)_q^{\otimes n}$ is the completion of $\ell^2(X)^{\otimes_{\rm alg} n}$ by the $q$-inner product. 
For $\xi \in \ell^2(X)$, we denote by $\ell(\xi)$ and $r (\xi)$ creation operators. We put $\ell(x) := \ell (\delta_x)$ and $r(x) := r (\delta_x)$ for $x \in X$, and 
\begin{align*}
	&\cC_\ell := \mathrm{C}^* \{ \ell(x) \mid x\in X\},\quad \cC_r := \mathrm{C}^* \{ r(x) \mid x\in X \}\quad \text{(if $q\neq 1$)};\\
	&\cC_\ell := \mathrm{C}^* \{ e^{\ri(\ell(x)+\ell(x)^*) } \mid x\in X\} = \cC_r\quad \text{(if $q = 1$)}.
\end{align*}
Let $U^\pi \colon \Gamma \to \cU (\cF_q)$ be the associated unitary representation on $\cF_q$ given by 
	$$U^\pi_g =1_\Omega \oplus \bigoplus_{n \geq 1} \pi_g^{\otimes n},\quad g\in \Gamma .$$
Note that $\Ad U^\pi$ gives $\Gamma$-actions on $\cC_\ell$ and $\cC_r$. Finally assume that there is a bijection $I\colon X\to X$ such that $I^2=\id_X$ and $I\pi_g=\pi_g I$ for all $g\in \Gamma$. Define an anti-linear isometry
	$$I\colon \ell^2(X)\to \ell^2(X);\quad a\delta_x \mapsto \overline{a}\delta_{Ix}\quad \text{for }a\in \C,\ x\in X.$$
We extend it on $\cF_q$ (as an anti-linear isometry) by 
	$$ J_{q}(\xi_1\otimes \cdots \otimes \xi_n)= I\xi_n\otimes \cdots \otimes I\xi_1$$
and $J_q\Omega =\Omega$. We have $J_q^2=\id$. This $J_q$ should be understood as the modular conjugation arising from von Neumann algebras given in Subsection \ref{Fock spaces and associated von Neumann algebras}. Examples of such von Neumann algebras will be given in Subsection \ref{Examples} (the bijectioin $I\colon X\to X$ is indeed important).

To study bi-exactness on Fock spaces, we introduce semigroup structures. More precisely, we will define a semigroup $S$ satisfying
	$$\cF_q\otimes \ell^2(\Gamma) = \ell^2(S\rtimes \Gamma)$$
and Assumption \ref{assumption}.

\subsection*{The case of full Fock spaces}

We consider the case $\cF_0=:\cF_{\rm full}$. Each $\ell^2(X)^{\otimes n}$ in $\cF_{\rm full}$ has an identification
	$$\ell^2(X)^{\otimes n} = \ell^2(X^{n});\quad \delta_{x_1}\otimes \cdots \otimes \delta_{x_n} = \delta_{(x_1, \ldots , x_n)}.$$
So by using the set 
	$$ \sfX := \{ \star \} \sqcup \bigsqcup_{n \geq 1} X^n,$$
where $\{\star\}$ is the singleton, we have $\ell^2 (\sfX) = \cF_{\rm full}$ with $\C \delta_\star = \C \Omega$. This $\sfX$ has a semigroup structure by
	$$ xy:=(x_1,\ldots,x_n,y_1,\ldots,y_m),\quad \text{for }x=(x_1,\ldots,x_n), \ y=(y_1,\ldots,y_m)\in \sfX. $$
Note that $\sfX=\ast_{X}\N$ and  $\star$ is the unit. We define creation operators for $x\in \sfX$ by 
	$$ \ell(x):=\ell(x_1)\cdots \ell(x_n) ,\quad r(x):=r(x_n)\cdots r(x_1).$$
They satisfy 
	$$ \ell(x)\delta_y=\delta_{xy},\quad r(x)\delta_y=\delta_{yx} ,\quad x,y\in \sfX.$$
We extend $\pi\colon \Gamma \act X$ on $\sfX$ by
	$$\pi_g(x_1, \ldots , x_n):=  (\pi_g(x_1), \ldots , \pi_g(x_1)).$$
Then the unitary $U_g^\pi$ on $\cF_\full$ satisfies
	$$\delta_{(x_1, \ldots , x_n)}=\delta_{x_1}\otimes \cdots \otimes \delta_{x_n} \mapsto^{U^\pi_g} \delta_{\pi_g(x_1)}\otimes \cdots \otimes \delta_{\pi_g(x_n)}=\delta_{(\pi_g(x_1), \ldots , \pi_g(x_1))}.$$
We extend $I$ on $\sfX$ by 
	$$ I\colon \sfX \to \sfX ;\quad (x_1,\ldots,x_n)\mapsto (Ix_n,\ldots,Ix_1) .$$
It holds that $I\pi_g=\pi_gI$ and $I(xy)=I(y)I(x)$ for all $g\in \Gamma$ and $x,y\in \sfX$. As in Assumption \ref{assumption}, we have an associated map $J_\full := J_{\sfX}$ on $\cF_q=\ell^2(\sfX)$ and it holds that $J_\full = J_q$. They satisfy all the conditions in Assumption \ref{assumption}.

\subsection*{The case of symmetric Fock spaces}

We consider the case $\cF_\sym:=\cF_{1}(\ell^2(X))$. Recall that for $\xi_1,\dots,\xi_n,\eta_1,\dots,\eta_n \in \ell^2(X)$,
\[
\langle \xi_1 \otimes \cdots \otimes \xi_n, \eta_1\otimes \cdots \otimes \eta_m \rangle = \delta_{n,m} \sum_{\sigma \in \frakS_n} \langle{\xi_{\sigma(1)}\otimes \cdots \otimes \xi_{\sigma(n)}, \eta_1 \otimes \cdots \otimes \eta_n}\rangle.
\]
Let $\ell^2(X)^{\otimes n}_\sym$ denote the completion of  $\ell^2(X)^{\otimes_{\rm alg} n}$ and we have
	$$ \cF_\sym=\C\Omega \oplus \bigoplus_{n\geq 1} \ell^2(X)^{\otimes n}_\sym.$$
For each $n\geq 1$, consider the natural action of the symmetric group $\frakS_n$ on $X^n$ by permutations of indices. Put $Y_n:=X^n/\frakS_n$ and define
	$$\sfY := \{\star\} \sqcup \bigsqcup_{n=1}^\infty Y_n,$$
which has a natural surjection $Q\colon \sfX \to \sfY$. We write $[x]:=Q(x)$ for $x\in \sfX$. We make identifications for all $n\in \N$ by 
\[
\ell^2 (Y_n) \ni \delta_{[x]} \mapsto C_{[x]} (\delta_{x_1}\otimes\cdots \otimes \delta_{x_n})\in \ell^2(X)^{\otimes n}_\sym,
\]
where $C_{[x]}>0$ is a normalized constant. We compute $C_{[x]}$ as follows. For any $[x_1,\dots,x_n] \in Y_n$, up to permutation, it has a form
	$$\delta_{x_1}\otimes \cdots \otimes \delta_{x_n} = \delta_{z_1}^{\otimes k_1}\otimes \cdots \otimes \delta_{z_m}^{\otimes k_m},$$
where all $z_1,\ldots,z_m\in X$ are distinct (hence $k_1+\cdots+k_m=n$). It holds that 
\[
\| \delta_{x_1}\otimes \cdots \otimes \delta_{x_n} \|^2 = \sum_{\sigma \in\frakS_n } \langle\delta_{x_{\sigma(1)}}\otimes \cdots \delta_{x_{\sigma(n)}}, \delta_{x_1}\otimes \cdots \otimes \delta_{x_n} \rangle
=k_1!k_2! \cdots k_m!.
\]
We get $C_{[x]} = \frac{1}{\sqrt{k_1!\cdots k_m!}} $. We thus have the identification $\ell^2(\sfY)=\cF_\sym$.

We consider a semigroup structure on $\sfY$ by 
	$$ [x][y]=[x_1,\ldots,x_n,y_1,\ldots,y_m],\quad \text{for }x=(x_1,\ldots,x_n), \ y=(y_1,\ldots,y_m)\in \sfX. $$
This means $\sfY = \bigoplus_X \N$ and $\star$ is the unit. For any $[x]=[x_1,\ldots,x_n]\in Y_n$, define
	$$ \ell([x]):=\ell(x_1)\cdots \ell(x_n) ,\quad r([x]):=r(x_n)\cdots r(x_1)$$
(actually $\ell([x])=r([x])$). By Lemma \ref{const-sym-lem} below, there are $C_{x,y}$ for $x,y\in \sfY$  such that 
	$$ \ell(x)\delta_y=C_{x,y} \delta_{xy}=  r(x)\delta_y.$$
We can extend $\pi$ and $I$ on $\sfY$ in the same way as for $\sfX$, so that
\begin{align*}
	 U_g^\pi&\colon \ell^2(\sfY)\to \ell^2(\sfY);\quad   \delta_{[x_1, \ldots , x_n]} \mapsto \delta_{[\pi_g(x_1), \ldots , \pi_g(x_n)]}\\
	 (J_\sym:=)J_{\sfY}&\colon \ell^2(\sfY)\to \ell^2(\sfY);\quad \delta_{[x_1, \ldots , x_n]} \mapsto \delta_{[Ix_n, \ldots , Ix_1]}
\end{align*}
are defined. They satisfy conditions in Assumption \ref{assumption}. 

\begin{lem}\label{const-sym-lem}
	For any $x\in X=Y_1$ and $[y] \in Y_n$, there is a constant $ C_{x,[y]}$ such that
	$$ \ell(x) \delta_{[y]}  = C_{x,[y]}\delta_{[(x,y)]},\quad 1\leq C_{x,[y]}\leq \sqrt{n+1}.$$
For any $x,y\in \sfY$, there is a constant $1\leq C_{x,y}$ such that
	$$ \ell(x) \delta_{y}  = C_{x,y}\delta_{xy} = C_{x,y}\delta_{yx} = r(x) \delta_{y}.$$
\end{lem}
\begin{proof}
	We have only to prove the first half of the statement. By the above notation, we identify $\delta_{[y]}$ as $\delta_{z_1}^{\otimes k_1} \otimes \cdots \otimes \delta_{z_m}^{k_m}$. Then 
\begin{align*}
	\ell(x) \delta_{[y]} 
	&= (k_1 !\cdots k_m!)^{-1/2} \delta_x\otimes \delta_{z_1}^{\otimes k_1} \otimes \cdots \otimes \delta_{z_m}^{k_m} \\
	&= \frac{\|\delta_x\otimes \delta_{z_1}^{\otimes k_1} \otimes \cdots \otimes \delta_{z_m}^{k_m}\|}{(k_1 !\cdots k_m!)^{1/2} }\delta_{[(x,y)]} =: C_{x,[y]}\delta_{[(x,y)]}.
\end{align*}
If there is $i$ such that $x = z_i$, then it is straightforward to compute that $C_{x,[y]}^2 =  k_i+1$. If there are no such $i$, then $C_{x,[y]}=1$. We get the inequality.
\end{proof}

\subsection*{The case of anti-symmetric Fock spaces}

We consider the case $\cF_\anti:=\cF_{-1}(\ell^2(X))$. Recall that 
for $\xi_1,\dots,\xi_n,\eta_1,\dots,\eta_n \in \ell^2(X)$, 
\[
\langle \xi_1 \otimes \cdots \otimes \xi_n, \eta_1\otimes \cdots \otimes \eta_m \rangle = \delta_{n,m} \sum_{\sigma \in \frakS_n}(-1)^{i(\sigma)} \langle{\xi_{\sigma(1)}\otimes \cdots \otimes \xi_{\sigma(n)}, \eta_1 \otimes \cdots \otimes \eta_n}\rangle,
\]
where $i(\sigma)$ is the number of inversions. Let $\ell^2(X)^{\otimes n}_\anti$ denote the completion of $\ell^2(X)^{\otimes_{\rm alg} n}$ and we have
	$$ \cF_\anti=\C\Omega \oplus \bigoplus_{n\geq 1} \ell^2(X)^{\otimes n}_\anti.$$
For each $n\in \N$, define 
	$$Z_n:=\{[x]=[x_1,\ldots,x_n] \in Y_n\mid x_i\neq x_j \quad \text{for all }i\neq j\}$$
and put
	$$\sfZ := \{\star\} \sqcup \bigsqcup_{n=1}^\infty Z_n\subset \sfY.$$
We would like to make an identification 
	$$\ell^2(Z_n)\ni \delta_{[x]}=\delta_{[x_1,\ldots,x_n]}\mapsto \delta_{x_1}\otimes \cdots \otimes \delta_{x_n}\in \ell^2(X)^{\otimes n}_\anti.$$
It is however not well defined, since it depends on the order of $x_1,\ldots,x_n$. So we fix a section $s\colon \sfZ\to \sfY$ such that $s(Z_n)\subset X^n$, and then the assignment
	$$\ell^2(Z_n)\ni \delta_{[x]}\mapsto \delta_{x_1}\otimes \cdots \otimes \delta_{x_n}\in \ell^2(X)^{\otimes n}_\anti,\quad \text{if }s([x]) = (x_1,\ldots,x_n)$$
defines a well defined unitary. We get an identification $ \ell^2(\sfZ) = \cF_\anti$, which obviously depends on the choice of the section. 
Observe that for any $z\in Z_n$ and any its representative $x=(x_1,\ldots,x_n)$, there is $C_x\in \{1,-1\}$ such that
	$$ \delta_z =C_x( \delta_{x_1}\otimes \cdots \otimes \delta_{x_n}). $$
Note that $C_{\sigma(x)}= \mathrm{sgn}(\sigma)C_x$ for any $\sigma\in \mathfrak S_n$. 
Since $\sfZ\subset \sfY$, there is a natural inclusion $\ell^2(\sfZ)\subset \ell^2(\sfY)$ and we denote by $P_{\sfZ}$ the orthogonal projection onto $\ell^2(\sfZ)$. We will use the semigroup structure of $\sfY$ and make creation operators which are supported on $\sfZ$, as explained in the last part in Subsection \ref{Bi-exactness and condition (AO)}.

For the semigroup structure on $\sfZ$, we use the one from $\sfY$. Then $\sfZ$ is not a subsemigroup but we can use creation operators as follows. First if $x\in X$, then we have $\ell(x),r(x)$ on $\cF_\anti$, so we can induce them on $\ell^2(\sfZ)$. By Lemma \ref{const-anti-lem} below, there are $C_{x,z},C_{x,z}^r\in \{0,1,-1\}$ for $z\in \sfZ$ such that 
	$$ \ell(x)\delta_z=C_{x,z} \delta_{xz},\quad r(x)\delta_z=C_{x,z}^r\delta_{zx} .$$
Note that $C_{x,z}=0$ if and only if $xz\not\in \sfZ$. 
For any $z\in \sfY \setminus \sfZ$, we put $\ell(z)=r(z)=0$. For any $[x]\in  Z_n\subset \sfY$ with $s([x])=x=(x_1,\ldots,x_n)$ by the section $s$, we define 
	$$ \ell([x])=\ell(x_1)\cdots\ell(x_n),\quad r([x])=r(x_n)\cdots r(x_1) .$$
Observe that for any other representative $y\in X^n$ with $[x]=[y]$, we have $\ell([x]) = C_y \ell(y_1)\cdots \ell(y_n)$. This implies that for any $z_1,z_2\in \sfZ$, $\ell(z_1)\ell(z_2)=_\T \ell_{z_1,z_2}\ell(z_1z_2)$ (possibly $0=_\T 0$). The same argument holds for right creations. Finally we regard all $\ell(s),r(s)$ for $s\in \sfY$ as operators on $\ell^2(\sfY)$, which are supported on $\ell^2(\sfZ)$. We thus get a family of creation operators.

We next consider the unitary representation and the anti-linear map. Observe that the involution $I$ and the action $\pi \colon \Gamma \act \sfY$ naturally restrict to ones on $\sfZ$. 
Then $J_1$ and $U^\pi$ on $\ell^2(\sfY)$ given in the symmetric case naturally restrict to ones on $\ell^2(\sfZ)$, so we use them as our objects for $\sfZ$ (hence we write $J_{-1}:=J_1$). 
We have to show that they satisfy conditions in Assumption \ref{assumption}. 
For any $g\in \Gamma$ and $[x]\in Z_n$ with $x=(x_1,\ldots,x_n)\in X^n$,
\begin{align*}
	&\delta_{[x_1, \ldots , x_n]}=C_x(\delta_{x_1}\otimes \cdots \otimes \delta_{x_n}) \mapsto^{U^\pi_g} C_x (\delta_{\pi_g(x_1)}\otimes \cdots \otimes \delta_{\pi_g(x_n)})=C_x C_{\pi_g(x)}^{-1}\delta_{[\pi_g(x_1), \ldots , \pi_g(x_n)]};\\
	&\delta_{[x_1, \ldots , x_n]}=C_x(\delta_{x_1}\otimes \cdots \otimes \delta_{x_n}) \mapsto^{J_{-1}} C_x (\delta_{Ix_n}\otimes \cdots \otimes \delta_{Ix_1})=C_x C_{I(x)}^{-1}\delta_{[Ix_n, \ldots , Ix_1]}.
\end{align*}
This means that for any $g\in \Gamma$ and $[x]\in \sfZ\subset \sfY$,
	$$ U^\pi_g \delta_{[x]} =_\T \delta_{\pi_g([x])}\quad \text{and}\quad J_{-1} \delta_{[x]} =_\T \delta_{I([x])}.$$
For any $[x]\in \sfY\setminus \sfZ$, we trivially have $U_g^\pi\delta_{[x]}=\delta_{\pi_g([x])}$ and $J_{-1}\delta_{[x]}=\delta_{I([x])}$. 
Thus putting $J_{\sfY} := J_{-1}$ in this setting, they satisfy conditions in Assumption \ref{assumption} for $\sfY$ which are supported on $\sfZ\subset \sfY$ in the sense that
	$$ \sfZ =\{ z\in \sfY\mid \ell(z)\neq 0 \text{ or }\ell(Iz)\neq 0\} = \{ z\in \sfY\mid \ell(z)\neq 0\}. $$
We will use this semigroup structure for $\sfZ$.

\begin{lem}\label{const-anti-lem}
	For any $x\in X=Z_1$ and $[z] \in \sfZ$, there are $ C_{x,[z]},C_{x,[z]}^r\in \{0,1,-1\}$ such that
	$$ \ell(x) \delta_{[z]}  = C_{x,[z]}\delta_{[(x,z)]},\quad r(x) \delta_{[z]}  = C^r_{x,[z]}\delta_{[(z,x)]}.$$
For any $x,y\in \sfZ$, there are $C_{x,y},C_{x,y}^r\in \{0,1,-1\}$ such that
	$$ \ell(x) \delta_{y}  = C_{x,y}\delta_{xy},\quad r(x) \delta_{y}=C^r_{x,y}\delta_{yx} .$$
For $x,y\in \sfZ$, we have that $C_{x,y}=0$ (or $C_{x,y}^r=0$) if and only $xy\not\in \sfZ$. 
\end{lem}
\begin{proof}
	We have only to prove the first part of the statement. Take any representative $z=(z_1,\ldots,z_n)$ of $[z]$. If there is $i$ such that $x=z_i$, then $\ell(x)\delta_{[z]}=0$, so we can put $C_{x,[z]}=0$. If there is no such $i$, then $(x,z)$ defines an element in $\sfZ$, hence
\begin{align*}
	\ell(x) \delta_{[z]} 
	&= \ell(x) C_z(\delta_{z_1} \otimes \cdots \otimes \delta_{z_n}) \\
	&= C_z(\delta_x\otimes \delta_{z_1} \otimes \cdots \otimes \delta_{z_n}) \\
	&= C_z C_{xz}^{-1}\delta_{[(x,z)]}.
\end{align*}
Then we can put $C_{x,[z]}:=C_zC_{xz}^{-1}$ (which is well defined by $C_{\sigma(z)}= \mathrm{sgn}(\sigma)C_z$).
\end{proof}

\subsection*{Family $\cG$ for relativity}

We keep the action $\Gamma\act^\pi X$ and associated semigroups $\sfX$, $\sfY$ and the subset $\sfZ\subset \sfY$. We first define a family of subsets in $\sfY\rtimes \Gamma$ by
	$$ \cG_\sym:=\left\{ \bigcup_{\rm finite} s\Gamma t \;\middle|\; s,t\in \sfY\rtimes \Gamma\right\}=\left\{ \bigcup_{\rm finite} s\Gamma t \;\middle|\; s,t\in \sfY\right\}.$$
We then define
	$$ \cG_\full:=\{ Q^{-1}(E) \subset \sfX\rtimes \Gamma \mid E\in \cG_\sym \} ,$$
where $Q\colon \sfX\rtimes \Gamma \to \sfY\rtimes \Gamma$ is the canonical surjection $Q(x,g) =([x],g)$, and
	$$ \cG_\anti:=\left\{ E\cap \sfZ \mid E\in \cG_\sym \right\}.$$
Observe that $\cG_\sym$ and $\cG_\full$ satisfy the relativity condition in Assumption \ref{assumption}.

Recall the following lemma, see \cite[Lemma 3.6]{HIK20}.

\begin{lem}\label{length function lemma}
	Assume that $\Gamma$ is countable and that $\Gamma \act^\pi X$ has finite stabilizers and finitely many orbits. Then there is a function $| \cdot |_X\colon \Gamma/\Lambda \to \R_{\geq 0}$ and a proper symmetric length function $|\cdot|_{\Gamma/\Lambda}\colon \Gamma/\Lambda \to \R_{\geq 0}$, where $\Lambda:=\bigcap_i\Lambda_i \leq \Gamma$, such that \begin{itemize}
	\item[$\rm (i)$] $|g\cdot x|_X\leq |g\Lambda|_{\Gamma/\Lambda} + |x|_X$;
	\item[$\rm (ii)$] $\{x\in X \mid |x|_X \leq R \}$ is finite for all $R> 0$.
\end{itemize}
\end{lem}

From now on we assume that the assumption in Lemma \ref{length function lemma} is satisfied. Then using the functions $|\cdot|_X$ and $|\cdot|_{\Gamma}(:=|\cdot|_{\Gamma/\Lambda})$, we introduce the following functions on our semigroups. For any $x =(x_1,\dots,x_n) \in X^n$ and $g\in \Gamma$, define
	$$|(x,g)|_0 :=n,\quad |(x,g)|_1 := \sum_{i=1}^n\min\{|x_i|_X,|\pi_g^{-1}(x_i)|_X\},$$
and $|(\star,g)|_0=|(\star,g)|_1 = 0$. Since they do not depend on the order of $x$, we can regard them as functions on $Y_n\times \Gamma$ by $|([x],g)|_k :=|(x,g)|_k$ for $k=0,1$. We define functions on semigroups by, for any $(x,g)\in \sfX\rtimes \Gamma$, 
	$$|(x,g)|_{\full}:=|(x,g)|_0+|(x,g)|_1,\quad |([x],g)|_{\sym}:=|([x],g)|_0+|([x],g)|_1.$$
Put $|([x],g)|_{\anti}:=|([x],g)|_{\sym}$ for $([x],g)\in \sfZ\rtimes \Gamma$. Note that $|([x],g)|_{\sym}= |(x,g)|_{\full}$ for all $ (x,g)\in \sfX\rtimes \Gamma$.

\begin{lem}\label{length-convergence-lem}
	Let $\ast$ be {\rm full, sym,} or {\rm anti}. Then for any net $(z_\lambda)_\lambda $ in $\mathsf{X}_\ast \rtimes \Gamma$, we have
	$$ z_\lambda \to \infty/\cG_{\ast} \quad \Leftrightarrow \quad \lim_\lambda |z_\lambda|_{\ast} = \infty.$$
\end{lem}
\begin{proof}
	By the definition of $\cG_\ast$ and $|\cdot|_\ast$, we have only to prove this lemma for the case that $\ast$ is sym.

	($\Rightarrow$) Fix $R>0$ and we show that the set of all $([x],g)\in \sfY\rtimes \Gamma$ satisfying $|([x],g)|_{\sym}\leq R$ is small$/\cG_\sym$. Fix such $([x],g)$ and write $x=(x_1,\ldots,x_n)$. Note that $n=|(x,g)|_0\leq R$.

Observe that 
	$$\min\{|x_i|_X ,|\pi_g^{-1}(x_i)|_X\leq |(x,g)|_1\leq R,$$
so for each $i$, $|x_i|_X$ or $|\pi_g^{-1}(x_i)|_X$ is smaller than $R$. Since we can change the order of $x$ in $Y_n$, we may assume that 
	$$ |x_i|_X \leq R \quad (i=1,\ldots,k),\quad |\pi_g^{-1}(x_j)|_X\leq R  \quad (j=k+1,\ldots,n).$$
As an element in $\sfX\rtimes \Gamma$, we have
	$$ (x,g) = x_1\cdots x_n g = x_1\cdots x_k \, g \, \pi_g^{-1}(x_{k+1})\cdots\pi_g^{-1}(x_n)  \in B_R(X)^k\, \Gamma \, B_R(X)^{n-k},$$
where $B_R(X)=\{a\in X\mid |a|\leq R\}$, which is a finite set. This implies that any $([x],g)$ with $|([x],g)|_{\sym}\leq R$ is contained in
	$$ \bigcup_{n\leq R}\bigcup_{k=0}^{n} Q(B_R(X)^k\, \Gamma \, B_R(X)^{n-k}),$$
where $Q$ is the canonical surjection onto $\sfY\rtimes \Gamma$. This is small relative to $\cG_\sym$.

	($\Leftarrow$) Fix any $A\subset \sfY\rtimes \Gamma$ which is small$/\cG_\sym$. We have to prove $\sup_{a\in A}|a|_\sym<\infty$. Since $A$ is contained in 
	$$ \bigcup_{\rm finite}s\Gamma t,\quad s,t\in \sfY ,$$
we may assume $A=s\Gamma t$ for some $s,t\in \sfY$. Write $s=[x,\ldots,x_n],t=[y_1,\ldots,y_m]\in \sfY$. Fix any $g\in \Gamma$ and write $ sgt=([x_1,\ldots,x_n,\pi_g(y_1),\ldots,\pi_g( y_m)],g)$. We have 
\begin{align*}
	|sgt|_\sym
	&= n +m + \sum_{i=1}^{n}\min\{|x_i|_X,|\pi_g^{-1}(x_i)|_X\} + \sum_{j=1}^m \min\{|\pi_g(y_j)|_X,|\pi_g^{-1}(\pi_g(y_j))|_X\}\\
	&\leq n +m + \sum_{i=1}^{n}|x_i|_X + \sum_{j=1}^m |y_j|_X=:M.
\end{align*}
Since $M$ does not depend on $g\in \Gamma$, we get $\sup_{a\in A}|a|_\sym\leq M$.
\end{proof}

\subsection{The case of full Fock spaces}\label{Bi-exactness on full Fock spaces}

We use the framework given in Subsection \ref{Fock spaces as semigroups}, such as $\sfX,I,\cG_\full,U_g^\pi,\ell(x),r(x)$ satisfying Assumption \ref{assumption}. Our goal is to prove the following theorem.

\begin{thm}\label{prop-amenable-full}
	Assume that $\Gamma$ is exact and $\Gamma\act^\pi X$ has finite stabilizers and finitely many orbits. Then the left translation $\Gamma \act C(\partial \cG_\full)$ is amenable. In particular $\Gamma \act^\pi \sfX$ is bi-exact/$\cG_\full$.
\end{thm}

We define a map $\omega \colon \sfX\rtimes \Gamma \to \ell^1(X)^+$ by
\[
\omega (x,g) := \sum_{i=1}^n (|(x,g)|_0 + \min\{|x_{i}|_X,|\pi_g^{-1}(x_i)|_X\}) \delta_{x_i} \quad \text{for} \quad (x,g) \in \sfX\rtimes \Gamma,
\]
and $\omega (\star ,g)$ is any nonzero element. Observe that
	$$ \|\omega(x,g) \|_1 = \sum_{i=1}^n |(x,g)|_0 + \min\{|x_{i}|_X,|\pi_g^{-1}(x_i)|_X\} = |(x,g)|_0^2 + |(x,g)|_1.$$
Up to normalization, define
\[
\mu\colon \sfX\rtimes \Gamma \to \Prob(X);\quad \mu (x,g) := \frac{\omega(x,g)}{\|\omega(x,g)\|_1},\quad (x,g)\in \sfX\rtimes \Gamma.
\]
We can induce a ucp map by
	$$\mu^* \colon \ell^\infty (X) \to \ell^\infty (\sfX\rtimes \Gamma);\quad f\mapsto [\sfX\rtimes \Gamma \ni (x,g) \mapsto \langle f, \mu(x,g)\rangle ].$$

\begin{lem}\label{lem-equiv-comm-full}
	The following statements hold true.
\begin{enumerate}
	\item (Equivariance) 
For any $g,h \in \Gamma$ and $\varphi \in \ell^\infty (X)$, 
	$$\mu^* ( g\cdot \varphi) - \mu^*( \varphi )(g^{-1}\, \cdot \, h) \in c_0(\cG_\full).$$

	\item (Commutativity)
For any $x,y\in \sfX$ and $\varphi\in \ell^\infty(X)$,
\begin{align*}
	&\mu^* ( \varphi ) - \mu^* (\varphi) ( x \,\cdot\, y)\in c_0 (\cG_\full).
\end{align*}
\end{enumerate}
In particular, $\mu^*(\ell^\infty(X))\subset C(\overline{\cG_\full})$ and we have a $\Gamma$-equivariant ucp map
	$$q\circ \mu^*\colon \ell^\infty(X) \to C(\partial\cG_\full),$$
where $q\colon C(\overline{\cG_\full})\to C(\partial \cG_\full)$ is the quotient map.
\end{lem}
\begin{proof}
	1. For any $z\in \sfX\rtimes \Gamma$, it is easy to compute that
\begin{align*}
	 \mu^*( g\cdot \varphi )(z) -  \mu^*( \varphi )(g^{-1} zh) 
	&= \langle \varphi,[g^{-1} \cdot \mu(z) -\mu(g^{-1}zh) ]\rangle .
\end{align*}
So we have only to prove that for any $g,h \in \Gamma$,
	$$\lim_{\sfX\rtimes \Gamma \ni z\to \infty/\cG_\full}\|\mu (gzh) - g\cdot \mu(z) \|_1=0 .$$

\begin{claim}
For any $g,h\in \Gamma$ and $z\in \sfX\rtimes \Gamma$,
	$$ \|g\cdot \omega(z)-\omega(gz)\|_1\leq |z|_0|g|_\Gamma ,\quad \|\omega(z)-\omega(zh)\|_1\leq |z|_0|h|_\Gamma.$$
\end{claim}
\begin{proof}
If we write $z=(a,t)$, $a=(a_1,\ldots,a_n)$, then it is easy to compute that
\begin{align*}
	&g\cdot \omega(z)
	=\omega(a,t)(\pi_g^{-1}\, \cdot \, )
	=\sum_{i=1}^n (n + \min\{|a_{i}|_X ,|\pi_t^{-1}(a_i)|\}) \delta_{\pi_g(a_i)};\\
	&\omega(gzh)
	=\omega( ( \pi_g(a),gth) )
	=\sum_{i=1}^n (n + \min\{|\pi_g(a_i)|_X,|\pi_{th}^{-1}(a_i)|\}) \delta_{\pi_g(a_i)},
\end{align*}
hence
\begin{align*}
	\|g\cdot \omega(z)-\omega(gzh)\|_1
	&=\sum_{i=1}^n  \left|\min\{|a_{i}|_X ,|\pi_t^{-1}(a_i)|\}- \min\{|\pi_g(a_i)|_X,|\pi_{th}^{-1}(a_i)|\right|.
\end{align*}
Using the inequality 
	$$ |\min\{a,c\}-\min\{b,c\}| \leq |a-b|,\quad \text{for }a,b,c\geq 0,$$
for the case $h=e$, we have
\begin{align*}
	\|g\cdot \omega(z)-\omega(gz)\|_1
	\leq \ & \sum_{i=1}^n  \left||a_{i}|_X - |\pi_g(a_i)|_X\right|
	\leq n |g|_\Gamma = |(a,t)|_0|g|_\Gamma .
\end{align*}
For the case $g=e$, we have
\begin{align*}
	\|\omega(z)-\omega(zh)\|_1
	\leq \ & \sum_{i=1}^n  \left||\pi_t^{-1}(a_i)| - |\pi_{th}^{-1}(a_i)|\right|
	\leq n |h^{-1}|_\Gamma = |(a,t)|_0|h|_\Gamma .
\end{align*}
\end{proof}
We compute that
\begin{align*}
	\|g\cdot\mu(z) -\mu(gzh)\|_1 
	&\leq \left\|\frac{g\cdot\omega(z)}{\|\omega(z)\|_1} - \frac{\omega(gzh)}{\|\omega(z)\|_1}\right\|_1 +\left\|\frac{\omega(gzh)}{\|\omega(z)\|_1} - \frac{\omega(gzh)}{\|\omega(gzh)\|_1}\right\|_1\\
	&\leq \frac{1}{\|\omega(z)\|_1}\left\|g\cdot\omega(z) - \omega(gzh)\right\|_1 + \left|\frac{\|\omega(gzh)\|_1}{\|\omega(z)\|_1} - 1\right|\\
	&\leq \frac{2}{\|\omega(z)\|_1}\left\|g\cdot\omega(z) - \omega(gzh)\right\|_1.
\end{align*}
By the claim, we get
\begin{align*}
	\|g\cdot\mu(z) -\mu(gzh)\|_1 
	&\leq 2(|g|_\Gamma + |h|_\Gamma)\frac{|z|_0}{|z|_0^2+|z|_1}.
\end{align*}
Now we consider $z\to \infty/\cG_\full$, which is equivalent to $|z|_0+|z|_1\to \infty$ by Lemma \ref{length-convergence-lem}. It is straightforward to prove that the right hand side of this inequality converges to 0. Hence we finish the proof of item 1.

	2. As in the proof of item 1, we have only to show that 
	$$\lim_{\sfX\rtimes \Gamma \ni z\to \infty/\cG_\full }\|\mu (xzy) - \mu(z) \|_1=0.$$
To see this, we may assume $x,y\in X$. 

\begin{claim}
For any $x,y\in X$,
	$$\|\omega(xz) - \omega(z)\|_1\leq 2|z|_0 + 1 + |x|_X,\quad \|\omega(zy) - \omega(z)\|_1\leq 2|z|_0 + 1 + |y|_X.$$
\end{claim}
\begin{proof}
Write $z=(a,t)$ and $a=(a_1,\ldots,a_n)$. Using 
\begin{align*}
	\omega(xz)
	= \omega((xa,t))
	=(n+1 + \min\{ |x|_X,|\pi_t^{-1}(x)|_X \}) \delta_{x} + \sum_{i=1}^{n} (n+1 + \min\{ |a_i|_X,|\pi_t^{-1}(a_i)|_X \}) \delta_{a_i},
\end{align*}
it is easy to see that
\begin{align*}
	\|\omega(xz)-\omega(z)\|_1
	=|n+1 + \min\{ |x|_X,|\pi_t^{-1}(x)|_X \}| + |n|
	\leq 2|z|_0+1 + |x|_X .
\end{align*}
The same computation works for $y$ (or apply $I$).
\end{proof}
By the claim, a computation similar to item 1 shows
\begin{align*}
	\|\mu(z) -\mu(xzy)\|_1 
	\leq \frac{2}{|z|_0^2 + |z|_1}(4|z|_0 + 4 + |x|_X  + |y|_X).
\end{align*}
It is straightforward to see that the right hand side converges to 0 as $z\to \infty/\cG_\full$.

The last statement is trivial by the definition of $C(\overline{\cG_\full})$.
\end{proof}

\begin{proof}[Proof of Theorem \ref{prop-amenable-full}]
We have a $\Gamma$-equivariant ucp map $q\circ \mu^* \colon \ell^\infty (X) \to C(\partial\cG_\full)$. Since $\Gamma$ is exact and since $\Gamma \act^\pi X$ has finite stabilizers and finitely many orbits, $\Gamma \act \ell^\infty(X)$ is amenable. Then by the ucp map, we get that $\Gamma \act C(\partial \cG_\full)$ is amenable (e.g.\ \cite[Exercise 15.2.2]{BO08}).
\end{proof}


\subsection{The case of symmetric Fock spaces}\label{Bi-exactness on symmetric Fock spaces}

As in the case of full Fock spaces, we use the framework given in Subsection \ref{Fock spaces as semigroups}. Recall that we use algebras
	$$ \cC:=\mathrm{C}^*\{ e^{\ri W(x)}\mid x\in X \},\quad W(x):=\ell(x)+\ell(x)^* $$
and $\cC_r:=J_\sym \cC J_\sym =\cC$. Observe that $\Ad(U_g^\pi)$ for $g\in \Gamma$ defines a $\Gamma$-action $e^{\ri W(x)}\mapsto e^{\ri W(\pi_g(x))}$. We define C$^*$-algebras by
\begin{align*}
	&\ell^\infty(\overline{\cG_\sym}):=\{ f\in \ell^\infty(\sfY\rtimes \Gamma)\mid [\pi_\ell(a),f],\ [\pi_r(a),f]\in \K(\cG_\sym)\text{ for all }a\in \cC \};\\
	&C(\overline{\cG_\sym}):=\{ f\in \ell^\infty(\overline{\cG_\sym})\mid f-f(\, \cdot \, h)\in c_0(\cG_\sym) \text{ for all }h\in \Gamma\}.
\end{align*}
We will show that they contain $c_0(\cG_\sym)$, hence we can define boundary C$^*$-algebras equipped with actions
\begin{align*}
	\Ad(U_g^\pi\otimes \lambda_g\rho_h) &\colon \Gamma\times \Gamma\act \ell^\infty(\partial\cG_\sym):=\ell^\infty(\overline{\cG_\sym})/c_0(\cG_\sym);\\
	\Ad(U_g^\pi\otimes\lambda_g)&\colon \Gamma \act C(\partial\cG_\sym):=C(\overline{\cG_\sym})/c_0(\cG_\sym).
\end{align*}
Our goal is to prove the following theorem.
\begin{thm}\label{prop-amenable-sym}
	Assume that $\Gamma$ is exact and $\Gamma \act^\pi X$ has finite stabilizers and finitely many orbits. Then the action $\Gamma \act C(\partial \cG_\sym)$ is amenable. In particular, $\Gamma\times \Gamma\act \ell^\infty(\partial\cG_\sym)$ is amenable.
\end{thm}

Let $\omega \colon \sfX\rtimes \Gamma \to \ell^1(X)^+$ be given in Subsection \ref{Bi-exactness on full Fock spaces}. Observe that it induces maps 
\begin{align*}
	&\omega_\sym \colon \sfY\rtimes \Gamma \to \ell^1(X)^+;\quad \omega_\sym ([x],g):=\omega(x,g);\\
	&\mu_\sym\colon \sfY\rtimes \Gamma \to \Prob(X);\quad \mu_\sym ([x],g) := \frac{\omega_\sym([x],g)}{\|\omega_\sym([x],g)\|_1} = \mu (x,g).
\end{align*}
Define 
	$$\mu_\sym^* \colon \ell^\infty (X) \to \ell^\infty (\sfY\rtimes \Gamma);\quad f\mapsto [\sfY\rtimes \Gamma \ni ([x],g) \mapsto \langle f, \mu_\sym([x],g)\rangle ].$$

\begin{lem}\label{lem-equiv-comm-sym}
	The following statements hold true.
\begin{enumerate}
	\item (Equivariance) 
For any $g,h \in \Gamma$ and $\varphi \in \ell^\infty (X)$,
	$$\mu_\sym^* ( g\cdot \varphi) - \mu_\sym^*( \varphi )(g^{-1}\, \cdot \, h) \in c_0(\cG_\sym).$$

	\item (Commutativity) 
For any $x,y\in X$ and $\varphi\in \ell^\infty(X)$,
\begin{align*}
	&\pi_\ell(C_{x,\cdot})(\mu_\sym^* ( \varphi ) - \mu_\sym^* (\varphi) ( x \,\cdot\, ))\in c_0 (\cG_\sym);\\
	&\pi_r(C^r_{y,\cdot})(\mu_\sym^* ( \varphi ) - \mu_\sym^* (\varphi) ( \,\cdot\, y))\in c_0 (\cG_\sym).
\end{align*}
\end{enumerate}
\end{lem}
\begin{proof}
	1. As in the proof of Lemma \ref{lem-equiv-comm-full}, we have only to prove
	$$\lim_{\sfY\rtimes \Gamma \ni z\to \infty/\cG_\sym}\|\mu_\sym (gzh) - g\cdot \mu_\sym(z) \|_1=0 .$$
Let $Q\colon \sfX\rtimes \Gamma \to \sfY\rtimes \Gamma$ be the canonical surjection. For any $z\in \sfX\rtimes \Gamma$, we have
	$$\mu_\sym (gQ(z)h)=\mu (gzh),\quad g\cdot \mu_\sym(Q(z))=g\cdot \mu(z).$$
Since $z\to \infty/\cG_\full$ and $Q(z)\to \infty/\cG_\sym$ are equivalent by Lemma \ref{length-convergence-lem}, we get
	$$\lim_{Q(z)\to \infty/\cG_\sym}\|\mu_\sym (gQ(z)h) - g\cdot \mu_\sym(Q(z)) \|_1 = \lim_{z\to \infty /\cG_\full}\|\mu (gzh) - g\cdot \mu(z) \|_1=0,$$
where the last equality is by Lemma \ref{lem-equiv-comm-full}.

	2. As in the proof of Lemma \ref{lem-equiv-comm-full} and by using the inequality in Lemma \ref{const-sym-lem}, we have
\begin{align*}
	|\pi_\ell(C_{x,\cdot})(z)(\mu_\sym^*(\varphi)(z)-\mu_\sym^*(\varphi)(xz))|
	&\leq  \|\varphi\|_\infty \sqrt{|z|_0+1} \, \|\mu_\sym (z)-\mu_\sym (xz)\|_1.
\end{align*}
As in the proof of item 1, the right hand side coincides with the term by $\mu$, hence by the proof of item 2 in Lemma \ref{lem-equiv-comm-full}, we have (up to exchanging $z$ with $Q(z)$),
\begin{align*}
	\sqrt{|Q(z)|_0+1}\, \|\mu_\sym(Q(z)) -\mu_\sym(Q(xz))\|_1 
	&= \sqrt{|z|_0+1}\, \|\mu(z) -\mu(xz)\|_1 \\
	&\leq \frac{2\sqrt{|z|_0+1}}{|z|_0^2 + |z|_1}(4|z|_0 + 4 + |x|_X ).
\end{align*}
It is straightforward to show that the last term converges to 0 as $z\to \infty/\cG_\full$, which is equivalent to $Q(z)\to \infty/\cG_\sym$. The same argument works for $y$.
\end{proof}

We next transfer the commutativity condition in the previous lemma to $e^{\ri W(x)}$. 

\begin{lem}\label{lem-equiv-comm-sym2}
	The following statements hold true.
\begin{enumerate}
 	\item For any $x\in X$ and $t\in \R$,
	$$ \pi_\ell(e^{\ri t W(x)}),\pi_r(e^{\ri t W(x)})\in \mathrm{M}(\K(\cG_\sym)) .$$

	\item For any $x\in  X$ and $f\in c_0(\cG_\sym)$, the maps
	$$ \R\ni t \mapsto \pi_\ell(e^{\ri t W(x)})f,\ \pi_r(e^{\ri t W(x)})f\in \K(\cG_\sym) $$
are norm continuous.

	\item For any $x \in X$ and $b \in \ell^\infty(\sfY \rtimes \Gamma)$, if $[\pi_\ell(W(x)) ,b]\in \K(\cG_\sym)$, then
	$$ \pi_\ell(e^{\ri W(x)})  b\pi_\ell(e^{-\ri W(x)}) -b  \in \K(\cG_\sym).$$
In particular we can put $b=\mu^*_\sym(\varphi)$ for $\varphi\in \ell^\infty(X)$. 
\end{enumerate}
\end{lem}
\begin{proof}
	1. Up to $\Ad(J)$, we have only to see $\pi_\ell$. We will prove  $\pi_\ell(e^{\ri tW(x)})1_A\in \K(\cG_\sym) $ for any $A\subset \sfY\rtimes \Gamma$ which is small $/\cG_\sym$ and any $0<|t|<1/2$. Fix such $A$ and $t$. 

\begin{claim}
There is $B\subset \sfY\rtimes \Gamma$ which is small$/\cG_\sym$ such that
	$$(W(x)^n\otimes 1) 1_{A} = 1_B(W(x)^n\otimes 1) 1_{A}.$$
In particular they are contained in $\K(\cG_\sym)$.
\end{claim}
\begin{proof}
	We may assume $n=1$. For simplicity, assume $A=s\Gamma t$ for some $s,t\in \sfY$. The range of $(\ell(x)\otimes 1)1_A$ has a basis $\delta_{xsgt}$, while a basis from $(\ell(x)^*\otimes 1)1_A$ is by subwords from $sgt$ (for which the letter $x$ is removed). By counting  these words and since $\Gamma \act^\pi X$ has finite stabilizers, it is easy to find the desired $B$.
\end{proof}

\begin{claim}
	The sum $\sum_{n\geq 0}\frac{1}{n!}((\ri tW(x))^n\otimes 1) 1_{A}$ has absolute convergence in norm.
\end{claim}
\begin{proof}
	For $m\in \N$, let $P_{\leq m}$ be the orthogonal projection onto the subspace spanned by vectors having tensor length less than $m$ in $\cF_\sym$. Observe that $\|W(x)P_{\leq m}\|_\infty\leq 2\sqrt{m+1}$ for all $m\in \N$. Since $A$ is small$/\cG_\sym$, there is a large $m\in \N$ such that $1_A\leq P_{\leq m}\otimes 1_\Gamma$. It holds that
\begin{align*}
	 \|(W(x)^n\otimes 1)1_A\|_\infty 
	&\leq \|W(x)^nP_{\leq m}\|_\infty \\
	&= \|W(x)^{n-1} P_{\leq m+1}W(x)P_{\leq m}\|_\infty \\
	&\leq \|W(x)P_{\leq m+n-1}\|_\infty \cdots \|W(x)P_{\leq m+1}\|_\infty\|W(x)P_{\leq m}\|_\infty \\
	&\leq (2\sqrt{m+n}) \cdots (2\sqrt{m+2})(2\sqrt{m+1}) ,
\end{align*}
so that 
\begin{align*}
	\sum_{n\geq 0}\frac{t^n}{n!}\|(W(x))^n\otimes 1) 1_{A}\|_\infty
	&\leq \sum_{n\geq 0}(2t)^n\frac{(\sqrt{m+1})(\sqrt{m+2})\cdots (\sqrt{m+n})}{n!}.
\end{align*}
Since $2t<1$, it is easy to get the conclusion.
\end{proof}

Using these two claims, we get $e^{\ri tW(x)}1_A =\sum_{n\geq 0}\frac{1}{n!}((\ri tW(x))^n\otimes 1) 1_{A} \in \K(\cG_\sym)$. This shows $e^{\ri tW(x)}\in \mathrm{M}(\K(\cG_\sym))$ for all $|t|<1/2$, hence for all $t\in \R$.

	2. Up to $\Ad(J)$, we have only to see $\pi_\ell$. Put $F(t):=\pi_\ell(e^{\ri tW(x)})-1$ and we have only to show $\lim_{t\to0}\| F(t) f\|_\infty\to0$ for any $f\in c_0(\cG_\sym)$. By approximating $f$, we may assume that $f=1_A$, $A=s\Gamma t$ for some $s,t\in \sfY$. Write $s=[a_1,\ldots,a_n]$ and $t=[b_1,\ldots,b_m]$. Since 
	$$s\Gamma t=\{ s g t = (s\pi_g(t),g) \mid g\in \Gamma\},$$
we have $1_A = \sum_{g\in \Gamma} 1_{\{s\pi_g(t)\}}\otimes 1_{\{g\}}$.  It holds that 
\begin{align*}
	\|F(t)1_A \|_\infty
	&=\left\|\sum_{g\in \Gamma}[(e^{\ri tW(x)} -1)1_{\{s\pi_g(t)\}}] \otimes e_{g,g}\right\|_\infty\\
	&=\sup_{g\in \Gamma}\|(e^{\ri tW(x)} -1) 1_{\{s\pi_g(t)\}}\|_\infty
	=\sup_{g\in \Gamma}\|(e^{\ri tW(x)} -1) \delta_{s\pi_g(t)}\|_{\cF_\sym}\\
	&\leq \sup_{g\in \Gamma}\|(e^{\ri tW(x)} -1)\ell(\pi_g(t)) \delta_{s}\|_{\cF_\sym}=\sup_{g\in \Gamma}\|\ell(\pi_g(t))(e^{\ri tW(x)} -1) \delta_{s}\|_{\cF_\sym}.
\end{align*}
We have to show that the last term converges to 0 as $t\to 0$. To see this, observe that, since $\Gamma \act^\pi X$ has finite stabilizers, there is a finite set $E\subset \Gamma$ such that 
	$$ \pi_g(b_i)\not\in \{x,a_1,\ldots,a_n\},\quad \text{for all }i=1,\ldots,m\quad\text{and}\quad g\not\in E.$$
Then for any $g\not\in E$, since we know the exact value of $C_{\pi_g(b),\cdot}$ by (the proof of) Lemma \ref{const-sym-lem} and since $(e^{\ri tW(x)} -1) \delta_a$ is contained in the space spanned by vectors indexed by $\{x,a_1,\ldots,a_n\}$, we have
\begin{align*}
	\|\ell(\pi_g(b)) (e^{\ri tW(x)} -1) \delta_a\|_{\cF_\sym}
	&\leq\sqrt{2}\cdots \sqrt{m} \, \| (e^{\ri tW(x)} -1) \delta_a\|_{\cF_\sym} \quad \text{for all } g\not\in E.
\end{align*}
The last term goes to 0 as $t\to 0$. Since $E$ is finite, we get the conclusion.

	3. Up to $\Ad(J)$, we have only to see $\pi_\ell$. Put  $a:=W(x)\otimes 1$ and $\pi_\ell(e^{\ri W(x)})=e^{\ri a}$. Consider the map
	$$F\colon \R \to \K(\cG_\sym);\quad  F(t) :=e^{\ri ta}[b,a]e^{-\ri ta}.$$
Since $e^{\ri ta}\in \mathrm{M}(\K(\cG_\sym))$ by item 1 and $[b,a]\in \K(\cG_\sym)$ by assumption, $F(t)$ is indeed contained in $\K(\cG_\sym)$. By item 2, it is norm continuous.

Let $\cF_{\rm alg}\subset \cF_\sym$ be the image of the algebraic Fock space and fix $\xi,\eta\in \cF_{\rm alg}\otimes_{\rm alg} \ell^2(\Gamma)$. Observe that by Stone's theorem, the function $f\colon \R\ni t \mapsto \langle e^{\ri ta}\xi ,\eta \rangle\in \C$ is differentiable and $f'(t)=\langle \ri ae^{\ri ta}\xi,\eta\rangle$. Hence the function 
	$$g(t):=\langle e^{\ri ta}be^{-\ri ta} \xi,\eta\rangle=\langle be^{-\ri ta} \xi,e^{-\ri ta}\eta\rangle$$
is also differentiable and $g'(t)= \ri\langle F(t) \xi,\eta\rangle$. For any $s>0$, we have
\begin{align*}
	\int_{0}^s \ri\langle F(t) \xi,\eta\rangle dt= \int_{0}^s g'(t) dt
	= g(s)-g(0)
	&= \langle (e^{\ri sa}be^{-\ri sa} -b)\xi,\eta\rangle .
\end{align*}
Since $F(t)$ is norm continuous, the integral $\int_0^s F(t)dt \in \K(\cG_\sym)$ is defined, so that 
	$$e^{\ri sa}be^{-\ri sa} -b = \ri \int_0^s F(t)dt\in \K(\cG_\sym)$$
for any $s>0$. This is the conclusion. The last statement follows by Lemma \ref{lem-equiv-comm-sym} and \ref{lem-basic-semigp}.
\end{proof}

\begin{proof}[Proof of Theorem \ref{prop-amenable-sym}]

Thanks to the previous two lemmas, we have $c_0(\cG_\sym)\subset \ell^\infty(\overline{\cG_\sym})$ and $\mu^*_\sym(\ell^\infty(X))\subset C(\overline{\cG_\sym})$. In particular, there is a $\Gamma$-equivariant ucp map
	$$q\circ \mu^*_\sym\colon \ell^\infty(X) \to C(\partial\cG_\sym),$$
where $q\colon C(\overline{\cG_\sym})\to C(\partial \cG_\sym)$ is the quotient map. We then follow the proof of Theorem \ref{prop-amenable-full} and Lemma \ref{bi-exact-lem}.
\end{proof}

\subsection{The case of anti-symmetric Fock spaces}\label{Bi-exactness on anti-symmetric Fock spaces}

We keep the framework given in Subsection \ref{Fock spaces as semigroups}. We prove the following Theorem. Here we are using the framework explained in Subsection \ref{Bi-exactness and condition (AO)} for the subset $\sfZ\subset \sfY$.

\begin{thm}\label{prop-amenable-anti}
	Assume that $\Gamma$ is exact and $\Gamma \act^\pi X$ has finite stabilizers and finitely many orbits. Then the left translation $\Gamma \act C(\partial \cG_\anti)$ is amenable. In particular $\Gamma \act^\pi \sfZ$ is bi-exact/$\cG_\anti$.
\end{thm}

Since $\sfZ\rtimes \Gamma=\sfZ\times \Gamma\subset \sfY\rtimes \Gamma$, by restriction, we can define
\begin{align*}
	\mu_\anti=\mu_\sym|_{\sfZ\rtimes \Gamma}\colon \sfZ\rtimes \Gamma \to \Prob(X);\quad 
	 \mu_\anti^* \colon \ell^\infty (X) \to \ell^\infty (\sfZ\rtimes \Gamma).
\end{align*}
To prove Theorem \ref{prop-amenable-anti}, as in the proof of Theorem \ref{prop-amenable-full}, we have only to show the following lemma.

\begin{lem}\label{lem-equiv-comm-anti}
	The following conditions hold true.
\begin{enumerate}
	\item (Equivariance) 
For any $g,h \in \Gamma$ and $\varphi \in \ell^\infty (X)$,
	$$\mu_\anti^* ( g\cdot \varphi) - \mu_\anti^*( \varphi )(g^{-1}\, \cdot \, h) \in c_0(\cG_\anti).$$

	\item (Commutativity)
For any $x,y\in X$ and $\varphi\in \ell^\infty(X)$,
\begin{align*}
	&\pi_\ell(C_{x,\cdot})(\mu_\anti^* ( \varphi ) - \mu_\anti^* (\varphi) ( x \,\cdot\, ))\in c_0 (\cG_\anti);\\
	&\pi_r(C^r_{y,\cdot})(\mu_\anti^* ( \varphi ) - \mu_\anti^* (\varphi) ( \,\cdot\, y))\in c_0 (\cG_\anti).
\end{align*}
\end{enumerate}
In particular, $\mu_\anti^*(\ell^\infty(X))\subset C(\overline{\cG_\anti})$ and we have a $\Gamma$-equivariant ucp map
	$$q\circ \mu^*\colon \ell^\infty(X) \to C(\partial\cG_\anti),$$
where $q\colon C(\overline{\cG_\anti})\to C(\partial \cG_\anti)$ is the quotient map. 
\end{lem}
\begin{proof}
	1. This is a special case of item 1 in Lemma \ref{lem-equiv-comm-sym}.

	2. We see only the case for $x\in X$. As in the proof of Lemma \ref{lem-equiv-comm-full}, we have for $z=(a,g)\in \sfZ\rtimes \Gamma$,
\begin{align*}
	&|\pi_\ell(C_{x,\cdot})(z)(\mu_\anti^*(\varphi)(z)-\mu_\anti^*(\varphi)(xz))|
	\leq  \|\varphi\|_\infty |C_{x,a}| \|\mu_\anti (z)-\mu_\anti (xz)\|_1.
\end{align*}
Hence we have only to show $\lim_{(\sfZ\rtimes \Gamma)_x \ni z\to \infty/\cG_\anti }\|\mu_\anti (xz) - \mu_\anti(z) \|_1=0$, where $ (\sfZ\rtimes \Gamma)_x$ is the set of all $z=(a,g)\in \sfZ\rtimes \Gamma$ such that $C_{x,a}\neq 0$ (which means $xz \in \sfZ\rtimes \Gamma$). As in the proof of Lemma \ref{lem-equiv-comm-sym}, this convergence follows by (the proof of) Lemma \ref{lem-equiv-comm-full}.
\end{proof}

\begin{rem}\label{independence of section}\upshape
	We have proved the bi-exactness of $\Gamma \act^\pi \sfZ$. We are interested in creation operators on $\cF_\anti$ and we will use the bi-exactness via the unitary $\cF_\anti \simeq \ell^2(\sfZ)$, which is given by a fixed section in Subsection \ref{Fock spaces as semigroups}. In this respect, it is worth mentioning that the associated C$^*$-algebras on $\cF_\anti\otimes \ell^2(\Gamma)$ do not depend on the choice of the section. Indeed, it is easy to see that the inclusion $\ell^\infty(\sfZ)\subset \B(\cF_\anti)$ does not depend on the section  (actually this coincides with the one given in \cite[Proposition 3.1]{HIK20}). Hence positions of
	$$c_0(\cG_\anti)\subset \ell^\infty(\sfZ\rtimes \Gamma),\quad \K(\cG_\anti)$$
in $\B(\cF_\anti \otimes \ell^2(\Gamma))$ are unique. Also $U_g^\pi$ for $g\in \Gamma$, $\ell(x),r(x)$ for $x\in X$, and $\cC_\ell,\cC_r$ are defined originally on $\cF_\anti$, hence they do not depend on the section. Finally the operators $\ell([x]),r([x])$ for $x\in \sfZ$ \textit{depend} on the section, but the dependence is only its sign, hence the algebra $C(\overline{\cG_\anti})$ does not depend on the section. Thus we can regard $c_0(\cG_\anti)\subset C(\overline{\cG_\anti})$ as subalgebras in $\B(\cF_\anti \otimes \ell^2(\Gamma))$, and we have the amenability of the action on the quotient algebra $C(\partial \cG_\anti)$. 

We finally note that the boundary constructed in \cite[Subsection 3.2]{HIK20} is different for ours. In \cite{HIK20}, we considered a quotient by $\cK$, which \textit{contains} our compact $\K(\cG_\anti)$. In this sense, our result is better and the associated condition (AO) is a stronger result since it uses the quotient algebra by a smaller compact operators. We note that the boundary in \cite{HIK20} is inspired by the construction in \cite{Oz04}, while ours is by the one in \cite{Oz08}.
\end{rem}

\section{Other examples}
\label{Other examples}

\subsection{Free products}

To study free wreath product groups, we first see infinite free product groups. The proof is a straightforward adaptation of known techniques.

\begin{prop}\label{prop-infinite-free-product}
	Let $I$ be a countable set and $\Gamma_i$ for $i\in I$ countable discrete groups. We denote by $\Gamma= \ast_{i\in I}\Gamma_i $ the free product group. Then $\Gamma$ is bi-exact if and only if so are all $\Gamma_i$'s.
\end{prop}
\begin{proof}
	The only if part is trivial, so we assume $\Gamma_i$ is bi-exact for all $i\in I$. For each $i\in I$, put 
	$$C(\overline{\Gamma}_i) = \{f\in \ell^\infty(\Gamma_i)\mid f(\,\cdot \, h) - f\in c_0(\Gamma_i) \}$$
and then $\Gamma_i \act C(\overline{\Gamma}_i)$ is amenable. In particular $\cA_i:= \mathrm{C}^*\{C(\overline{\Gamma}_i), C_\lambda^*(\Gamma_i)\}\simeq C(\overline{\Gamma}_i)\rtimes \Gamma $ is nuclear and contains all compact operators. By the proof of \cite[Lemma 2.4]{Oz04}, $\ast_{i\in F}\cA_i$ is nuclear for all finite subsets $F\subset I$, hence $\cA:=\ast_{i\in I}\cA_i$ is nuclear. Then define the following $\ast$-homomorphism
	$$ \nu\colon \cA \otm J\cA J \to \B(\ell^2(\Gamma))/\K(\ell^2(\Gamma));\quad x\otimes JyJ\mapsto xJyJ + \K(\ell^2(\Gamma)).$$
(It is well defined, since $[\cA,J\cA J]\subset \K(\ell^2(\Gamma))$.) Since $\cA \otm J\cA J$ is nuclear, the restriction of $\nu$ on $C_\lambda^*(\Gamma)\otm C_\rho(\Gamma)$ is nuclear. By the lifting theorem, $\nu$ has a ucp lift, hence $\Gamma$ is bi-exact by \cite[Lemma 15.1.4]{BO08}.
\end{proof}

\subsection{Free wreath products}

Let $\Gamma,\Delta$ be countable discrete groups and $\Gamma \act \cI$ an action on a countable set. Define
\begin{align*}
	\text{wreath product}\quad &\Delta\wr_{\cI} \Gamma := \Delta_{\cI} \rtimes \Gamma, \quad \text{where }\Delta_{\cI}:=\bigoplus_{\cI} \Delta;\\
	\text{free wreath product}\quad &\Delta\wr_{\cI}^{\free} \Gamma := \Delta_{\cI}^\free \rtimes \Gamma, \quad \text{where }\Delta_{\cI}^\free := \ast_{\cI} \Delta.
\end{align*}
Here both $\Gamma$-actions are given by coordinate shifts along $\Gamma \act \cI$. In this subsection, we study some boundary amenability for $\Gamma\act \Delta_\cI^\free$ given in Theorem \ref{prop-biexact-free}. We will then obtain the following characterization of bi-exactness for free wreath product groups. The proof uses also Proposition \ref{bi-exact-Fock} which will be proved in Subsection \ref{Condition (AO) and rigidity}\label{Condition (AO) and rigidity}.

\begin{thm}
	Assume that $\Gamma \act \cI$ has finite stabilizers and finitely many orbits. Then $\Delta \wr^\free_{\cI} \Gamma$ is bi-exact if and only if so are $\Delta$ and $\Gamma$.
\end{thm}
\begin{proof}
	The only if part is trivial. By Theorem \ref{prop-biexact-free} and Proposition \ref{bi-exact-Fock}, if $\Gamma,\Delta$ are bi-exact, $L(\Delta\wr^\free_\cI\Gamma)$ with its group C$^*$-algebra satisfies condition (AO) with a ucp lift. By \cite[Lemma 15.1.4]{BO08}, $\Delta\wr^\free_\cI\Gamma$ is bi-exact.
\end{proof}

This theorem should be compared to the following Ozawa's theorem. 

\begin{thm}[{\cite{Oz04,BO08}}]
	Assume that $\Gamma \act \cI$ has finite stabilizers and finitely many orbits. Then $\Delta\wr_{\cI} \Gamma$ is bi-exact if and only if $\Delta$ is amenable and $\Gamma$ is bi-exact.
\end{thm}

We start the proof of Theorem \ref{prop-biexact-free} below. We denote by $\Delta_i$ the $i$-th copy of $\Delta$ in $\Delta_\cI^\free$. Any non trivial $y\in \Delta_\cI^\free$ has a unique word decomposition
	$$ y= y_{i_1}\cdots y_{i_n},\quad y_{i_k}\in \Delta_{i_k},\quad  i_1\neq i_2,\ldots,i_{n-1}\neq i_n .$$
Set $\supp(y):=\{i_1,\ldots,i_n\}\subset I$. 
For any finite subsets $E\subset \Delta$, $F\subset \cI$ and $n\in \N$, we define $A(E,F,n)\subset \Delta_\cI^\free\rtimes  \Gamma$ in the following way. For any $y\in \Delta_\cI^\free$ and $g\in \Gamma$, $(y,g)$ is in $A(E,F,n)$ if $y=e$ or the word decomposition $y=y_{i_1}\cdots y_{i_m}$ satisfies
	$${\rm (a)} \  m\leq n\qquad {\rm (b)} \ y_{i_k}\in E \text{ for all } i_k \qquad {\rm (c)} \ \mathrm{supp}(y)\subset F\cup gF.$$
It is then straightforward to check the following conditions: for  any finite subsets $E,E'\subset \Delta$, $F,F'\subset \cI$, $n,m\in \N$, $x\in \Delta_j$, and $g\in \Gamma$,
\begin{itemize}
	\item $A(E,F,n)\cup A(E',F',m) \subset A(E\cup E',F\cup F',\max\{n,m\})$;

	\item $A(E,F,n)^{-1} =A(E^{-1},F,n)$;

	\item $x A(E,F,n)\subset A(E\cup\{x\}\cup x E,F \cup \{j\},n+1)$,

$  A(E,F,n)x \subset A(E\cup\{x\}\cup E x,F \cup \{j\},n+1)$;

	\item $g A(E,F,n) \subset A(E,F \cup g^{-1}F,n),\quad  A(E,F,n)g \subset A(E,F \cup gF,n)$.

\end{itemize}
Hence if we define
	$$ \cG_\free:=\{ A\subset \Delta_\cI^\free \rtimes \Gamma \mid A\subset A(E,F,n) \text{ for some finite subsets }E,F \text{ and }n\in \N\}, $$
then it is globally preserved by finite unions, group inverse, and left-right translations. We can consider the relativity condition in $\ell^\infty(\Delta_\cI^\free \rtimes \Gamma)$ given by $\cG_\free$. Note that $\cG_\free$ is not generated by subgroups in $\Delta_\cI^\free \rtimes \Gamma$, hence it is not covered by Ozawa's argument in \cite{BO08}. Define
\begin{align*}
	 C(\overline{\cG_\free}):=\{f\in \ell^\infty(\Delta_{\cI}^\free\rtimes \Gamma) \mid f(x \, \cdot h) - f\in c_0(\cG_\free)\quad \text{for all }x\in \Delta_\cI^\free,\ h\in \Delta_\cI^\free \rtimes\Gamma\}
\end{align*}
and $C(\partial \cG_\free):=C(\overline{\cG_\free})/c_0(\cG_\free)$. It admits the $\Gamma$-action arising from the left translation.

Our goal in this subsection is to show the following theorem.

\begin{thm}\label{prop-biexact-free}
	Assume that $\Gamma$ and exact and $\Gamma \act \cI$ has finite stabilizers and finitely many orbits. Then the $\Gamma$-action $\Gamma \act C(\partial \cG_\free)$ is amenable. In particular $\Gamma \act \Delta_\cI^\free$ is bi-exact$/\cG_\free$.
\end{thm}

From now on, assume that $\Gamma \act \cI$ has finite stabilizers and finitely many orbits. 
By Lemma \ref{length function lemma}, take proper length functions $|\,\cdot \, |_\Gamma$ ($=|\,\cdot \, |_{\Gamma/\Lambda})$ in the statement) and $|\,\cdot \, |_\cI$. Also take $|\cdot|_\Delta$ for $\Delta$. 
Then we define a length in $\Delta_\cI^\free\rtimes \Gamma$ by, for any $y\in \Delta_\cI^\free$ with decomposition $y=y_{i_1}\cdots y_{i_n}$ and $g\in \Gamma$, 
\begin{align*}
	 |(y,g)|_{\free}
	:= \sum_{i\in \mathrm{supp}(y)} \min\{|i|_{\cI}, |g^{-1}\cdot i|_{\cI}\} + \sum_{k=1}^n |y_{i_k}|_{\Delta} .
\end{align*}
Put $|(e,g)|_{\free}=0$. The next lemma is useful.

\begin{lem}
The following conditions hold true.
\begin{enumerate}
	\item For any net $(z_\lambda)_\lambda$ in $\Delta_{\cI}^\free\rtimes \Gamma$,
	$$z_\lambda \to \infty/\cG_\free \quad \Leftrightarrow \quad |z_\lambda|_{\free} \to \infty.$$

	\item $\lim_{(y,g)\to \infty/\cG_\free}\frac{|\mathrm{supp}(y)|}{|(y,g)|_{\free}} = 0$.

\end{enumerate}
\end{lem}
\begin{proof}
	1. $(\Rightarrow)$ Take any $N\in \N$ and we show that  $\{z\in \Delta_\cI^\free\rtimes \Gamma\mid |z|_{\free}\leq N\}$ is small$/\cG_\free$. Take any such $z=(y,g)$ and write $y=y_{i_1}\cdots y_{i_n}$. Then since
	$$\sum_{k=1}^n |y_{i_k}|_{\Delta}\leq |(y,g)|_\free\leq N ,$$
and since $1\leq |y_{i_k}|_{\Delta}$ for all $i_k$ by construction, it follows that (a) $n\leq N$ and (b) $y_{i_k}\in B_N(\Delta)$ for all $i_k$, where $B_N(\Delta)$ is the $N$-ball in $\Delta$ with respect to $|\, \cdot \,|_\Delta$. 
For any $i\in \mathrm{supp}(y)$, since 
	$$ \min\{|i|_{\cI}, |g^{-1}\cdot i|_{\cI} \}\leq |z|_\free \leq  N ,$$
it follows that (c) $i\in B_N(\cI)\cup gB_N(\cI)$, where $B_N(\cI)$ is the $N$-ball in $\cI$. Thus $z$ is contained in $A(E,F,N)$ for $E=B_N(\Delta)$ and $F=B_N(\cI)$, both of which are finite sets by properness.

	$(\Leftarrow)$ Fix any $A(E,F,n)$ and we show $\sup_{z\in A(E,F,n)} |z|_\free<\infty$. If $yg=y_{i_1}\cdots y_{i_m}g\in A(E,F,n)$, then
\begin{align*}
	|(y,g)|_{\free} 
	&= \sum_{i\in \mathrm{supp}(y)} \mathrm{min}\{|i|_{\cI}, |g^{-1}\cdot i|_{\cI} \} + \sum_{k=1}^m |y_{i_k}|_{\Delta} \\
	&\leq \sum_{i\in \mathrm{supp}(y)} \sup_{j\in F}|j|_\cI+ \sum_{k=1}^m \sup_{x\in E}|x|_{\Delta} \\
	&\leq  2|F|\sup_{j\in F}|j|_\cI+ n\sup_{x\in E}|x|_{\Delta} <\infty.
\end{align*}

	2. We prove a claim.
\begin{claim}
For any $(y,g)\in \Delta_\cI^\free\rtimes \Gamma$ and $C>0$,
	$$|(y,g)|_{\free}\leq C\, |\mathrm{supp}(y)|\quad \Rightarrow \quad |(y,g)|_{\free}\leq 4C |B_{2C}(\cI)|.$$
\end{claim}
\begin{proof}
We can assume $|(y,g)|_{\free}\neq 0$. 
We have only to show $|\mathrm{supp}(y)|\leq 4 |B_{2C}(\cI)|$, so suppose $|\mathrm{supp}(y)|> 4 |B_{2C}(\cI)|$. 
Put $B_g:=B_{2C}(\cI) \cup gB_{2C}(\cI) $ and observe that
	$$|B_g|\leq 2 |B_{2C}(\cI)|< \frac{1}{2}|\mathrm{supp}(y)|.$$
Then since $\mathrm{supp}(y)\setminus B_g\neq \emptyset$ and since  $\mathrm{min}\{|i|_{\cI}, |g^{-1}\cdot i|_{\cI} \} > 2C$ for all $i\not\in B_g$, it follows that
\begin{align*}
	C\, |\mathrm{supp}(y)|
	\geq |(y,g)|_{\free}
	\geq \sum_{i\in \mathrm{supp}(y)\setminus B_g}  \mathrm{min}\{|i|_{\cI}, |g^{-1}\cdot i|_{\cI} \}
	> 2C|\mathrm{supp}(y)\setminus B_g|.
\end{align*}
We get $|\mathrm{supp}(y)|> 2 |\mathrm{supp}(y)\setminus B_g|$. This implies
\begin{align*}
	|\mathrm{supp}(y)|
	&= |\mathrm{supp}(y)\setminus B_g| + | B_g \cap \supp(y)|\\
	&< \frac{1}{2}|\mathrm{supp}(y)| + \frac{1}{2}|\mathrm{supp}(y)| = |\mathrm{supp}(y)|.
\end{align*}
This is a contradiction.
\end{proof}

Suppose by contradiction that there is a net $z_\lambda =(y_\lambda,g_\lambda)$ such that $z_\lambda \to \infty/\cG_\free$ and $|\mathrm{supp}(y_\lambda)||z_\lambda|_{\free}^{-1}\not \to 0$. Up to a subnet, we can assume there is $\delta>0$ such that 
	$$\frac{|\mathrm{supp}(y_\lambda)|}{|z_\lambda|_{\free}}\geq \delta\quad \text{for all }\lambda.$$
Then the claim for the case $C:=\delta^{-1}$ shows that $|z_\lambda|_{\free}\leq 4C |B_{2C}(\cI)|$ for all $\lambda$. This is a contradiction, since $|z_\lambda|_\free\to \infty$ by item 1.
\end{proof}

We define a map 
\begin{align*}
	\omega \colon \Delta_{\cI}^\full \rtimes \Gamma \to \ell^1(\cI)^+;\quad 
	\omega (y,g) = m(y,g) + a(y),
\end{align*}
where for $y=y_{i_1}\cdots y_{i_n}\in \Delta_\cI^\free$ and $g\in \Gamma$,
\begin{align*}
	m(y,g) = \sum_{i\in \mathrm{supp}(y)} \min\{|i|_{\cI}, |g^{-1}\cdot i|_{\cI}\}\delta_{i}\quad \text{and}\quad
	a(y)=\sum_{k=1}^n |y(i_k)|_{\Delta} \delta_{i_k},
\end{align*}
and $m(e,g)=0=a(e)$. We note that
\begin{itemize}
	\item $\|\omega(y,g)\|_1 = \|m(y,g)\|_1 + \|a(y)\|_1 = |(y,g)|_\free$;
	\item $a (y) = \sum_{k=1}^n a (y_{i_k})$ and $\|a (y)\|_1 = \sum_{k=1}^n \|a(y_{i_k})\|_1$ \quad for $y=y_{i_1}\cdots y_{i_n}$;
	\item $\omega (y^{-1},g)=\omega (y,g)$.
\end{itemize}
Up to normalization, we define
\begin{align*}
	\mu \colon \Delta_{\cI}^\full \rtimes \Gamma \to \mathrm{Prob}(\cI);\quad 
	\mu(z) = \frac{\omega (z)} {\|\omega(z)\|_1}.
\end{align*}

\begin{lem}
The following conditions hold true.
\begin{enumerate}
	\item For any $z=(y,h)\in \Delta_{\cI}^\free \rtimes \Gamma$ and $g\in \Gamma$,
	$$\|g\cdot \omega(z) - \omega(gz)\|_1 \leq |g|_\Gamma |\mathrm{supp}(y)|,\quad \|\omega(z) - \omega(zg)\|_1 \leq |g|_\Gamma |\mathrm{supp}(y)|.$$

	\item For any $z\in \Delta_{\cI}\rtimes \Gamma$ and $x\in \Delta_j$ for $j\in \cI$,
	$$\|\omega(xz) - \omega(z)\|_1 \leq \|\omega(x)\|_1,\quad \|\omega(zx) - \omega(z)\|_1 \leq \|\omega(x)\|_1.$$

\end{enumerate}
\end{lem}
\begin{proof}
For $g\in \Gamma$ and $y\in \Delta_\cI^\free$, we denote by $\pi_g(y)$ the associated action. If $y=y_{i_1}\cdots y_{i_n}$, then $\pi_g(y) = \pi_g(y_{i_1})\cdots \pi_g(y_{i_n})$ and each $\pi_g(y_{i_k})$ is $y_{i_k}$ as an element in $\Delta$ and is contained in $\Delta_{g\cdot i_k}$.

	1. By definition, we have
\begin{align*}
	\omega(gz) 
	&= \sum_{i\in \mathrm{supp}(y)} \min\{|g\cdot i|_{\cI}, |(gh)^{-1}\cdot g\cdot i|_{\cI}\} \delta_{g\cdot i} + \sum_{k=1}^n|y_{i_k}|_{\Delta} \delta_{g\cdot i_k};\\
	g\cdot \omega(z)  
	&= \sum_{i\in \mathrm{supp}(y)}\min\{|i|_{\cI}, |h^{-1}\cdot i|_{\cI}\} \delta_{g\cdot i} + \sum_{k=1}^n |y_{i_k}|_{\Delta} \delta_{g\cdot i_k}.
\end{align*}
It follows that
\begin{align*}
	\| g\cdot \omega(z) - \omega(gz)\|_1
	&=  \sum_{i\in \mathrm{supp}(y)} |\min\{|i|_{\cI}, |h^{-1}\cdot i|_{\cI}\}- \min\{|g\cdot i|_{\cI}, |h^{-1}\cdot i|_{\cI}\} |\\
	\ & \leq  \sum_{i\in \mathrm{supp}(y)}||i|_{\cI}-|g\cdot i|_{\cI}| 
	\leq |\mathrm{supp}(y)| |g|_\Gamma .
\end{align*}
The second inequality follows similarly.

	2. We see a claim.
\begin{claim}
For $z=(y,g)\in \Delta_\cI^\free\rtimes \Gamma$ and $x\in \Delta_j$, it holds that
	$$ \|m(xy,g) - m(y,g)\|_1\leq \|m(x,g)\|_1 ,\quad \|a(xy) - a(y)\|_1\leq \|a(x)\|_1.$$
\end{claim}
\begin{proof}
Write $y=y_{i_1}\cdots y_{i_n}$  and $ xz = x y_{i_1}\cdots y_{i_n} g$.
Observe that
\begin{itemize}
	\item $j\neq i_1$\quad $\Rightarrow$ \quad $a(xy) = a(x)+a(y)$;

	\item $j= i_1$\quad $\Rightarrow$ \quad $a(xy) = a(xy_{i_1})+a(y_{i_2}\cdots y_{i_n})$.
\end{itemize}
If $j\neq i_1$, then $\|a(xy) - a(y)\|_1 = \|a(x)\|_1$. If $j= i_1$, then
\begin{align*}
	\|a(xy)-a(y) \|_1
	&= \|a(xy_{i_1}) + a(y_{i_2}\cdots y_{i_n}) - a(y)\|_1\\
	&= \|a(xy_{i_1})  - a(y_{i_1})\|_1\\
	&= ||xy_{i_1}|_\Delta  - |y_{i_1}|_{\Delta} ||\leq |x|_\Delta=\|a(x)\|_1.
\end{align*}
Observe next that
\begin{itemize}

	\item $j\not\in \mathrm{supp}(y)$\quad $\Rightarrow$ \quad $m(xy,g) =m(x,g) + m(y,g)$;

	\item $j\in \mathrm{supp}(y)$\quad $\Rightarrow$ \quad $m(xy,g) = m(y,g)$.

\end{itemize}
If $j\not \in \mathrm{supp}(y)$, then $\|m(xy,g) -m(y,g)\|_1 = \|m(x,g)\|_1$. If $j\in \mathrm{supp}(y)$, then  $\|m(xy,g) - m(y,g)\|_1= 0\leq  \| m(x,g)\|_1$.
\end{proof}

By the claim, we get 
\begin{align*}
	\|\omega(xy,g) - \omega(y,g)\|_1
	&\leq \|m(xy,g) - m(y,g)\|_1 + \|a(xy) - a(y)\|_1\\
	&\leq \|m(x,g)\|_1 + \|a(x)\|_1 = \|\omega(x,g)\|_1.
\end{align*}
Since $\|\omega(x,g)\|_1\leq \|\omega(x)\|_1$, we get the first inequality. 

We see the second one. It is straightforward to see that $\|\omega(\pi_g(x),g)\|_1\leq \|\omega(x,e)\|_1$ for any $g\in \Gamma$ and $x\in \Delta_\cI^\free$. Then since $\omega(y^{-1},g) = \omega(y,g)$,
\begin{align*}
	\| \omega(y(g\cdot x),g) - \omega(y,g) \|_1
	&= \| \omega((g\cdot x^{-1})y^{-1},g) - \omega(y^{-1},g) \|_1\\
	&\leq \| \omega((g\cdot x^{-1}),g) \|_1\\
	&\leq \| \omega(x^{-1},e) \|_1 = \| \omega(x,e) \|_1.
\end{align*}
This is the conclusion.
\end{proof}

Now following the same arguments as in the Fock spaces, we get the following lemma. Theorem \ref{prop-biexact-free} obviously follows by this lemma.

\begin{lem}
The following statements hold true.
\begin{enumerate}

	\item (Equivariance) For any $g,h\in \Gamma$ and $\varphi\in \ell^\infty(\cI)$,
	$$    \mu^*(g\cdot \varphi) - \mu^*(\varphi)(g^{-1}\, \cdot \, h)\in c_0(\cG_\free).$$

	\item (Commutativity) For any $x,y\in \Delta_{\cI}^\free$ and $\varphi\in \ell^\infty(\cI)$,
	$$ \mu^*(\varphi)(x\, \cdot \, y) - \mu^*(\varphi)\in c_0(\cG_\free).$$
\end{enumerate}
In particular, $\mu^*(\ell^\infty(\cI))\subset C(\overline{\cG_\free})$ and we have a $\Gamma$-equivariant ucp map
	$$q\circ \mu^*\colon \ell^\infty(\cI) \to C(\partial\cG_\free),$$
where $q\colon C(\overline{\cG_\free})\to C(\partial \cG_\free)$ is the quotient map.
\end{lem}

\section{Application to rigidity of von Neumann algebras}\label{Application to rigidity of von Neumann algebras}

\subsection{Examples}\label{Examples}

We introduce concrete examples of crossed product $B\rtimes_{\rm red} \Gamma\subset M\rtimes \Gamma$, for which our rigidity results are applied. We always assume that $\Gamma \act^\pi X$ is an action 
of a countable exact group $\Gamma$ on a set $X$, which has finite stabilizers and finitely many orbits. Consider $\ell^2(X)$ and associated Fock spaces $\cF_\ast$, where $\ast$ is full, sym, or anti. We have $U_g^\pi$ and $J_\ast$ acting on $\cF_\ast$.

\subsection*{Gaussian algebras acting on $\cF_\sym$}

Assume $I=\id_X$ in this case. As in Subsection \ref{Bi-exactness on symmetric Fock spaces}, we put $B:=\mathrm{C}^*\{e^{\ri W(x)}\mid x\in X\}$ and $M:=B''$. We are interested in the inclusion $B\rtimes_{\rm red}\Gamma \subset M\rtimes \Gamma$. Note that $\Gamma\act M$ is the generalized Bernoulli action with diffuse base, arising from $\Gamma \act X$. 
Observe that $(\cF_\sym \otimes \ell^2(\Gamma),J_\sym)$ is the standard representation with the modular conjugation of $M\rtimes \Gamma$, and $U_g^\pi$ is the standard implementation of the action $\Gamma \act M$.

\subsection*{(Free) Araki--Woods algebras acting on $\cF_\full$ and $\cF_\anti$}

Let $\R\to \mathcal O(H_\R)$ be any strongly continuous representation on a real Hilbert space $H_\R$. Put $H:=H\otimes_\R \C$ and let $I_\R$ be the involution for $H_\R \subset H$. Using the infinitesimal generator $A$, as in Subsection \ref{Fock spaces and associated von Neumann algebras}, consider $\Gamma_q(H_\R,U)$ for $q=0,-1$ and associated objects
\begin{align*}
	&j(\xi)=\frac{\sqrt{2}}{\sqrt{1+A^{-1}}}\, \xi ,\quad W(j(\xi)):=\ell(j(\xi)) + \ell(j(I\xi))^*,\quad \xi\in H;\\
	& B:=\mathrm{C}^*\{ W(j(\xi)) \mid \xi \in H\} \subset \mathrm{W}^*\{ W(j(\xi)) \mid \xi \in H\}  = \Gamma_q(H_\R,U)=:M.
\end{align*}
It holds that $L^2(M)=\cF_*$, where $\ast$ is full or anti. Assume that there is an identification $\ell^2(X)=H$ such that 
\begin{itemize}
	\item the involutions $I$ from $\ell^2(X)$ and $I_\R$ from $H_\R \subset H$ coincide;
	\item $\pi_g$ and $U_t$ on $\ell^2(X)=H$ commute for all $g\in \Gamma$ and $t\in \R$.
\end{itemize}
In this assumption, $\Gamma \act^\pi H$ induces a vacuum state preserving action $\Gamma \act M$. Then it is straightforward to check that $U_g^\pi$ and $J_\ast$ arising from the structure of $\ell^2(X)$ coincide with the standard implementation and the modular conjugation arising from that of $H$ and $(H_\R,U_t)$.

We give concrete examples satisfying the above assumptions. Let $\Gamma \act^\pi \cI$ be any action of a countable group on a set $\cI$ which has finite stabilizers and finitely many orbits. Fix $\lambda\geq 1$ and put $H(\lambda):=\ell^2(\cI)=:H(\lambda^{-1})$ and $H:=H(\lambda )\oplus H(\lambda^{-1})=\ell^2(\cI^1\sqcup \cI^2)$, where $\cI=\cI^1=\cI^2$. 
Put $X:=\cI^1\sqcup \cI^2$ and let $i_k\colon \cI\to X$ be the $k$-th embedding for $k=1,2$. Extend $\pi$ on $X$ diagonally and define a bijection $I$ on $X$ by the flip, that is, $Ii_1(x)=i_2(x)$ and $Ii_2(x)=i_1(x)$ for all $x\in \cI$. We consider $U\colon \R \act H$ by $U_t\xi = \lambda^{\ri t}\xi $ for $\xi\in H(\lambda)$ and $U_t\xi = \lambda^{-\ri t}\xi $ for $\xi\in H(\lambda^{-1})$. Observe that we can regard $U$ as an almost periodic representation arising from $H_\R \subset H$. In this case, the involution $I_\R$ for $H_\R \subset H$ coincides with $I$ for $\ell^2(X)$ and $\pi_g$ commutes with $U_t$ for all $g\in \Gamma$ and $t\in \R$. Hence they satisfy the above assumptions. More generally we can take a finite direct sum of such $H$.

\subsection*{Free wreath products}

Let $\Delta$ be a countable bi-exact group. Putting $\cI:=X$, consider the free wreath product $\Delta\wr^\free_\cI \Gamma = \Delta_\cI^\free \rtimes \Gamma$. We put $B:=C_\lambda^\ast (\Delta_\cI^\free)\subset L(\Delta_\cI^\free)=:M$ and 
	$$B\rtimes_{\rm red}\Gamma = C_\lambda^\ast (\Delta_\cI^\free \rtimes \Gamma)\subset L (\Delta_\cI^\free \rtimes \Gamma) = M\rtimes \Gamma .$$
The associated action $\Gamma \act \Delta_\cI^\free$ and the group inverse naturally induce $U_g^\pi$ and $J_\free$ on $\ell^2(\Delta_\cI^\free)$ in this setting. We also have the family $\cG_\free$ and put $\sf X_{\free}:= \Delta_{\cI}^\free$.

\subsection{Condition (AO) and rigidity}\label{Condition (AO) and rigidity}

	Let $B\rtimes_{\rm red}\Gamma \subset M\rtimes \Gamma$ be an inclusion given in the last subsection. We put $B_r:=J_\ast BJ_\ast$ with $\Gamma$-action by $\Ad(U_g^\pi)$ and consider $B_r\rtimes_{\rm red}\Gamma$. 

Summarizing all our previous results, we obtain the following (relative) condition AO and solidity.

\begin{prop}\label{bi-exact-Fock}
	Keep the notation from Subsection \ref{Examples}. Then the following $\ast$-homomorphism is min-bounded, nuclear, and has a ucp lift:
\begin{align*}
	\nu\colon (B\rtimes_{\rm red}\Gamma) \otimes_{\rm alg} (B_r \rtimes_{\rm red}\Gamma) \to \mathrm{M}(\K(\cG_\ast))/\K(\cG_\ast);\quad x\otimes y\mapsto \pi_\ell(x)\pi_r(y)+\K(\cG_\ast).
\end{align*}
It $\Gamma$ is bi-exact, we can replace $\K(\cG_\ast)$ with $\K(L^2(M)\otimes \ell^2(\Gamma))$, hence $M\rtimes \Gamma $ satisfies condition $\rm (AO)$ with a ucp lift.
\end{prop}
\begin{proof}
	Let $\mathcal E$ be the family of all finite subsets in $\Gamma$ and define $\cF$ as in Lemma \ref{intersection-AO-lem}. Observe that $\K(\cG_\ast\cap \cF) = \K(\ell^2(\sf X_\ast\rtimes \Gamma))$.

	We first consider the case for $\cF_\anti$. In this case, since $\Gamma \act^\pi\sfZ$ is bi-exact$/\cG_\anti$ by Theorem \ref{prop-amenable-anti}, $B$ is nuclear, and since $[B,B_r]=0$, we can directly apply Proposition \ref{bi-exact-prop}. For the last part of the statement, we can use Lemma \ref{intersection-AO-lem} for $\cF$.

We next consider the case for $\cF_\full$. We first claim that $[\pi_\ell(a),\pi_r(b)]\subset \K(\ell^2(\sfX\rtimes\Gamma))$ for all $a\in \cC_\ell$ and $b\in \cC_r$. We may assume $a=\ell(x)$ and $b=r(y)$ for some $x,y\in X$. Since $[\ell(x)^*,r(y)]= \delta_{x,y} P_{\C\Omega}$, where $P_{\C\Omega}$ is the projection onto $\C\Omega$, we have
\begin{align*}
	[\pi_\ell(\ell(x)^*),\pi_r(r(y))]
	&= \sum_{h\in \Gamma} \delta_{x,\pi_h(y)}P_{\C\Omega}\otimes e_{h,h}=:f\in \ell^\infty(\sfX\rtimes \Gamma).
\end{align*}
It satisfies $f(s,g) =\delta_{s,\star} \delta_{x,\pi_g(y)}$, hence it is contained in $c_0(\sfX\rtimes \Gamma)$ (since $\pi$ has finite stabilizers). The claim is proven. 
Then since $\cC_\ell,\cC_r$ are nuclear, $\nu_{\cG_\full}$ and $\nu_\cF$ in Proposition \ref{bi-exact-prop} and Lemma \ref{intersection-AO-lem} are min-bounded and nuclear (because they are defined at the level of $\cC_\ell,\cC_r$). We get the conclusion.

We consider the case for $\cF_\sym$. In this case, we can not apply Proposition \ref{bi-exact-prop} and Lemma \ref{intersection-AO-lem}, but we can follow the same proofs in this setting by the amenability of $\Gamma \act C(\partial \cG_\sym)$ in Theorem \ref{prop-amenable-sym}. They are all straightforward, hence we omit it.

We consider the case of free wreath products. Let $\cA:=\ast_{i\in \cI}\cA_i$ be a nuclear C$^*$-algebra containing $B$ used in the proof of Proposition \ref{prop-infinite-free-product}. Then since each $\cA_i$ is contained in $\mathrm{C}^*\{C_\lambda^\ast (\Delta_i),\ \ell^\infty(\Delta_i)\}$, $\cC$ is contained in $\mathrm{C}^*\{C_\lambda^\ast (\Delta_\cI^\free),\ \ell^\infty(\Delta_\cI^\free)\}$. 
We claim that $[\pi_\ell(a),\pi_r(JbJ)] \in \K(\ell^2(\Delta_\cI^\free\rtimes \Gamma))$ for all $a,b\in \cA$. 
We may assume $a\in \cA_i$ and $b\in \cA_j$ for some $i,j$. Then
\begin{align*}
	[\pi_\ell(a),\pi_r(JbJ)]
	= \sum_{h\in \Gamma} [ a,  J \alpha_h(b) J]\otimes e_{h,h},
\end{align*}
where $\alpha_h(b)$ is $b$ contained in $\cA_{h\cdot j}$. As in the proof in Proposition \ref{prop-infinite-free-product}, $ [a,J \alpha_h(b) J] =0$ if $h\cdot j \neq i$ and $[a,J \alpha_h(b) J]$ is compact if $h\cdot j=i$ (which happens for finitely many $h\in \Gamma$). We conclude $[\pi_\ell(a),\pi_r(JbJ)]$ is compact and the claim is proven. 
Now we can define $\nu_{\cG_\full}$ and $\nu_\cF$ in Proposition \ref{bi-exact-prop} and Lemma \ref{intersection-AO-lem} at the level of $\cA$ and $J\cA J$, hence they are min-bounded and nuclear. We get the conclusion.
\end{proof}

\begin{thm}\label{solid-thm1}
	Keep the setting from Subsection \ref{Examples}. Then the action $\Gamma \act M$ is solid. Further $M\rtimes \Gamma$ is solid if $\Gamma$ is bi-exact.
\end{thm}
\begin{proof}
	Observe that $\K(\cG_\ast)$ is contained in 
	$$ \cK:=\{ \sum_{g\in \Gamma} x_g\otimes e_{g,g}\mid x_g\in \K(L^2(M)) ,\ \sup_g\|x_g\|_\infty <\infty\}. $$
Then one can follow Ozawa's proof \cite{Oz04}, see \cite[Theorem 4.7]{HIK20}. 
\end{proof}

\begin{rem}\upshape\label{solid-remark}
	We finally explain relationship between our Theorem \ref{solid-thm1} and known results in Popa's deformation/rigidity theory.
\begin{itemize}
	\item The case of symmetric Fock spaces (i.e.\ actions on Gaussian algebras) was proved in \cite{Bo12} in a more general setting. 

	\item The case of full symmetric Fock space (i.e.\ actions on free Araki--Woods factors) was studied in \cite{HS09,Ho12,HT18}. When $\Gamma$ is amenable, the solidity is easily deduced from \cite[Theorem D]{HT18}. When $\Gamma$ is not amenable, the solidity result in Theorem \ref{solid-thm1} is not discussed in these articles. 

	\item The case of anti-symmetric Fock space (i.e.\ actions on Araki--Woods factors) is not studied, so Theorem \ref{solid-thm1} provides new examples.

	\item For the case of free wreath product groups, the solidity is not known, so Theorem \ref{solid-thm1} provides new examples.

\end{itemize}
Thus Theorem \ref{solid-thm1} provides some new examples of solid actions and solid factors. We however strongly believe that appropriate adaptations of known techniques in Popa's deformation/rigidity theory should be applied to above new examples. 
Therefore we do not emphasize that they are really new examples, but we do emphasize that our proofs involve boundary amenability which have independent interests.
\end{rem}

\small{

}
\end{document}